\documentclass[notitlepage,11pt,reqno]{amsart}
\usepackage[foot]{amsaddr}
\RequirePackage{amssymb,nicefrac,bm,upgreek,mathtools,verbatim,enumerate}
\usepackage[final]{hyperref}
\RequirePackage[mathscr]{eucal}
\RequirePackage{dsfont}
\RequirePackage[normalem]{ulem}

\usepackage{graphicx}
\usepackage{wrapfig}
\usepackage{caption}
\usepackage{subcaption}

\usepackage[margin=1in]{geometry}
\allowdisplaybreaks
\hypersetup{
colorlinks=true,
  citecolor=mblue,
  linkcolor=mblue,
  frenchlinks=false,
  pdfborder={0 0 0},
  naturalnames=false,
  hypertexnames=false,
  breaklinks}
\usepackage[msc-links,abbrev,nobysame]{amsrefs}
\newcommand{\ttup}[1]{\textup{(}#1\textup{)}}
\newcommand{\stkout}[1]{\ifmmode\text{\sout{\ensuremath{#1}}}\else\sout{#1}\fi}
\usepackage{color}
\definecolor{dmagenta}{rgb}{.4,.1,.5}
\definecolor{dblue}{rgb}{.0,.0,.5}
\definecolor{mblue}{rgb}{.0,.0,.7}
\definecolor{ddblue}{rgb}{.0,.0,.4}
\definecolor{dred}{rgb}{.7,.0,.0}
\definecolor{dgreen}{rgb}{.0,.5,.0}
\definecolor{Eeom}{rgb}{.0,.0,.5}
\definecolor{cm}{cmyk}{1,.0,.0,.0}

\numberwithin{equation}{section}

\theoremstyle{plain}
\newtheorem{theorem}{Theorem}[section]
\newtheorem{lemma}{Lemma}[section]

\newtheorem{proposition}{Proposition}[section]

\theoremstyle{definition}

\newtheorem{definition}{Definition}[section]
\newtheorem{notation}{Notation}[section]

\theoremstyle{remark}
\newtheorem{remark}{Remark}[section]

\usepackage[capitalize,nameinlink]{cleveref}

\crefname{section}{Section}{Sections}
\crefname{subsection}{Subsection}{Subsections}
\crefname{condition}{Condition}{Conditions}
\crefname{hypothesis}{Hypothesis}{Conditions}
\crefname{assumption}{Assumption}{Assumptions}
\crefname{lemma}{Lemma}{Lemmas}
\crefname{corollary}{Corollary}{Corollaries}
\crefname{theorem}{Theorem}{Theorems}
\crefname{claim}{Claim}{Claims}
\crefname{fact}{Fact}{Facts}
\crefname{notation}{Notation}{Notations}

\Crefname{figure}{Figure}{Figures}

\crefformat{equation}{\textup{#2(#1)#3}}
\crefrangeformat{equation}{\textup{#3(#1)#4--#5(#2)#6}}
\crefmultiformat{equation}{\textup{#2(#1)#3}}{ and \textup{#2(#1)#3}}
{, \textup{#2(#1)#3}}{, and \textup{#2(#1)#3}}
\crefrangemultiformat{equation}{\textup{#3(#1)#4--#5(#2)#6}}%
{ and \textup{#3(#1)#4--#5(#2)#6}}{, \textup{#3(#1)#4--#5(#2)#6}}%
{, and \textup{#3(#1)#4--#5(#2)#6}}

\Crefformat{equation}{#2Equation~\textup{(#1)}#3}
\Crefrangeformat{equation}{Equations~\textup{#3(#1)#4--#5(#2)#6}}
\Crefmultiformat{equation}{Equations~\textup{#2(#1)#3}}{ and \textup{#2(#1)#3}}
{, \textup{#2(#1)#3}}{, and \textup{#2(#1)#3}}
\Crefrangemultiformat{equation}{Equations~\textup{#3(#1)#4--#5(#2)#6}}%
{ and \textup{#3(#1)#4--#5(#2)#6}}{, \textup{#3(#1)#4--#5(#2)#6}}%
{, and \textup{#3(#1)#4--#5(#2)#6}}

\crefdefaultlabelformat{#2\textup{#1}#3}

\newcommand{\process}[1]{{\{#1_t\}_{t\ge0}}}

\newcommand{\df}{\coloneqq}
\DeclareMathOperator{\Exp}{\mathbb{E}} 
\DeclareMathOperator{\Prob}{\mathbb{P}} 
\newcommand{\D}{\mathrm{d}} 
\newcommand{\E}{\mathrm{e}} 
\newcommand{\RR}{\mathbb{R}} 
\newcommand{\Rm}{{\mathbb{R}^m}} 
\newcommand{\NN}{\mathbb{N}} 
\newcommand{\ZZ}{\mathbb{Z}} 
\newcommand{\Ind}{\mathds{1}} 
\newcommand{\Act}{\mathbb{U}} 
\newcommand{\dd}{\mathfrak{d}}

\newcommand{\Lg}{{\mathcal{L}}}
\newcommand{\cL}{{\mathscr{L}}}  
\newcommand{\cI}{{\mathcal{I}}}
\newcommand{\cJ}{{\mathcal{J}}}
\newcommand{\nJ}{J} 

\newcommand{\cA}{\mathcal{A}}   
\newcommand{\sB}{{\mathscr{B}}} 
\newcommand{\cB}{{\mathcal{B}}} 
\newcommand{\Cc}{\mathcal{C}}    
\newcommand{\cD}{\mathcal{D}}    
\newcommand{\cG}{{\mathcal{G}}}  
\newcommand{\cE}{{\mathcal{E}}}  
\newcommand{\sF}{{\mathcal{F}}}  
\newcommand{\cK}{{\mathcal{K}}} 
\newcommand{\sS}{{\mathscr{S}}}  
\newcommand{\sX}{{\mathscr{X}}}  
\newcommand{\cZn}{{\mathcal{Z}^n}} 
\newcommand{\hcZn}{{\Hat{\mathcal{Z}}^n}} 
\newcommand{\tcZn}{{\widetilde{\mathcal{Z}}^n}} 
\newcommand{\tfZn}{{\widetilde{\mathfrak{Z}}^n}} 

\newcommand{\Usm}{\mathfrak{U}_{\mathsf{sm}}}  

\newcommand{\transp}{^{\mathsf{T}}} 
\newcommand{\order}{{\mathscr{O}}} 

\newcommand{\Lyap}{{\mathscr{V}}}
\newcommand{\veo}{{\varepsilon}}

\newcommand{\mx}{_\mathsf{max}}
\newcommand{\mn}{_\mathsf{min}}

\newcommand{\Sobl}{\mathscr{W}_{\mathrm{loc}}}  

\newcommand{\abs}[1]{\lvert#1\rvert}
\newcommand{\norm}[1]{\lVert#1\rVert}
\newcommand{\babs}[1]{\bigl\lvert#1\bigr\rvert}

\newcommand{\babss}[1]{\biggl\lvert#1\biggr\rvert}
\newcommand{\bnorm}[1]{\bigl\lVert#1\bigr\rVert}

\DeclareMathOperator{\diag}{diag}
\DeclareMathOperator{\trace}{trace}

\usepackage{accents}
\newlength{\dhatheight}

\newcommand{\ttl}{\Large 
Uniform stability of some large-scale parallel server networks }
\begin{document}
\title[Uniform stability of large-scale parallel server networks]{\ttl}

\author[Hassan Hmedi]{Hassan Hmedi$^*$}
\address{$^*$ Department of Electrical and Computer Engineering,
The University of Texas at Austin,
Austin, TX~~78712}
\email{$\lbrace$hmedi,ari$\rbrace$@utexas.edu}

\author[Ari Arapostathis]{Ari Arapostathis$^*$}

\author[Guodong Pang]{Guodong Pang$^\dag$}
\address{$^\dag$ Department of Computational and Applied Mathematics,
George R. Brown College of Engineering,
Rice University, Houston, TX~~77005}
\email{gdpang@rice.edu}

\begin{abstract}
In this paper we study the uniform stability properties of two classes of parallel server networks with multiple classes of jobs and multiple server pools of a tree topology.
These include a class of networks with a single non-leaf server pool, such as the `N' and `M' models, and  networks of any tree topology with
class-dependent service rates. 
We show that with $\sqrt{n}$ safety staffing, and no abandonment, in the Halfin--Whitt regime, the diffusion-scaled controlled queueing processes are exponentially ergodic and their invariant probability distributions are tight, \emph{uniformly} over all stationary Markov controls.
We use a unified approach in which the same Lyapunov function is used in the study of the prelimit and diffusion limit.

A parameter called \emph{the spare capacity (safety staffing)
of the network} plays a central role in characterizing the stability results: the parameter being positive is necessary and sufficient that the limiting diffusion is uniformly exponentially ergodic over all stationary Markov controls. 
We introduce the concept of ``\emph{system-wide work conserving policies}", which are defined as policies that minimize the number of idle servers at all times. This is stronger than the so-called joint work conservation. We show that, provided the spare capacity parameter is positive,  the diffusion-scaled processes are geometrically ergodic and the invariant distributions are tight, uniformly over all ``system-wide work conserving policies". 
In addition, when the spare capacity is negative we show that the diffusion-scaled processes are transient under any stationary Markov control, and when it is zero, they cannot be positive recurrent.

\end{abstract}



\maketitle

\section{Introduction}

Large-scale parallel server networks have been the subject of intense study,
due to their use in modeling a variety of systems including telecommunications,
data centers, customer services and manufacturing systems;
see, e.g., \cite{AAM, GKM03,brown-et-al,garnett-mandel-reiman,
AIMMTYT,Shietal,van2017economies,van2018scalable}.  
In such networks, there are multiple classes of jobs and multiple server pools where each job class can be served by a subset of server pools while each server pool can serve a subset of job classes, thus requiring optimal routing and scheduling decisions.  
Many of these systems operate in the so-called the 
Halfin--Whitt regime (or Quality-and-Efficiency-Driven (QED) regime
\cite{halfin-whitt,W92,borst2004dimensioning}),
where the arrival rates and the numbers of servers grow large as the scale of
the system grows, while the service rates remain fixed in such a way that the system
becomes critically loaded.

Ensuring stability of these systems through allocating available resources by
means of adjusting controller parameters is of great importance.
Existing work in the literature has addressed the following important questions:
\begin{enumerate}
\item[(i)]
Uniform stability of the multiclass single-pool ``V" network. 
The study in \cite{GS12} focused on the prelimit diffusion-scaled process and showed
that, with square-root safety staffing in the single-pool of servers, the invariant
probability distributions under all work-conserving 
scheduling policies are tight, and have a uniform exponential tail when
the model has no abandonment 
(or a sub-Gaussian tail with abandonment). 
In \cite{AHP18}, a unified approach with a common Lyapunov function is developed to 
establish a Foster-Lyapunov equation 
for both the diffusion limit and the diffusion-scaled processes,
which shows that the associated invariant
probability measures have exponential tails,
uniformly over the scale of the network, and over all stationary
(work-conserving) Markov controls.

\item[(ii)]
Stability of the `N' network under a static priority scheduling policy. 
With safety staffing in one server pool and no abandonment, Stolyar \cite{Stolyar-15b}
employed a integral type of Lyapunov function and established the tightness of
stationary distributions of the diffusion-scaled process
(there is no analysis of the rate of convergence though).

\item[(iii)]
Counterexamples for stability of multi-class multi-pool  networks.
Stolyar and Yudovina \cite{Stolyar-Yudovina-13} 
showed that the stationary distributions of the diffusion-scaled processes may not be 
tight in these regimes under a natural load balancing 
scheduling policy, ``Longest-queue freest-server" (LQFS-LB)
(also true in the underloaded regime). 

\item[(iv)]
Stability of multi-class multi-pool networks with pool-dependent service rates under
the LQFS-LB policy  \cite{Stolyar-Yudovina-13}.
We also refer the reader to \cite{Stolyar-15,Stolyar-Yudovina-12},
even though these concern the underloaded case.

\item[(v)]
Stability of multiclass multi-pool networks under a family of
Markov policies.
In \cite{AP18,AP19}, it is shown that a 
class of state-dependent policies, referred to as
\emph{balanced saturation policies} (BSP) are stabilizing for the prelimit 
diffusion-scaled queueing process, when at least one abandonment rate is strictly 
positive.

\item[(vi)]
Stability of the limiting controlled diffusions for multiclass multi-pool networks
under a constant control.
Arapostathis and Pang \cite{AP16} developed a leaf elimination algorithm to
derive an explicit expression of the drift, and, consequently, by using the structural
properties of the drift, a static priority scheduling and routing control
is identified which stabilizes
the limiting diffusion, when at least one of the classes has a positive 
abandonment rate. 

\item[(vii)] Stabilizability of multiclass multi-pool networks of any tree topology without abandonment in the Halfin-Whitt regime. 
Hmedi, Arapostahis and Pang \cite{HAP-OR} identified a system-wide safety staffing parameter and showed that that parameter being positive is a necessary and sufficient condition for the network to be stabilizable, that is, there exists a scheduling policy under which the 
stationary distributions of the controlled diffusion-scaled queueing processes are tight over the size of the network. 
\end{enumerate}

The stability results in (v) and (vi) are used in the aforementioned
papers for  the study of ergodic control problems for multiclass multi-pool
networks.
In \cite{ABP15,AP16,AP18,AP19}, due to the lack of the ``uniform stability''
(also called  ``blanket stability") property, 
ergodic control problems were studied using a rather elaborate methodology.
The uniform stability properties established in
this paper render the ergodic control problem much simpler,
and it can be studied by applying the methodology
in  \cite[Chapter 3.7]{ABG12}.

Despite all the important results in (i)--(vi), the ergodic properties of multiclass
multi-pool networks in the Halfin--Whitt 
regime are far from being well understood.  
The stability analysis of multiclass multi-pool networks in the Halfin--Whitt regime is 
considerably more challenging than the  corresponding one for the `V'~network. 
The problem is particularly difficult when the system does not have abandonment.

Given the counterexamples in \cite{Stolyar-Yudovina-13}, uniform stability,
that is,
tightness of the invariant probability distributions,
does not hold for multiclass multi-pool networks of any tree topology.
In this paper we identify a large class of such networks
that are indeed uniformly stable: (a) networks with  one dominant server pool, 
that is, a single non-leaf server pool, which include the `N', `M' and 
generalized `N', `M'~networks with diameters equal to three or four,
and (b) networks with class-dependent service rates.
It might appear to the reader that the topology in (a) is restrictive.
One should note though that even for simple networks with two non-leaf
server pools \cite[Figure~2, p.~21]{Stolyar-Yudovina-13} the parameters can
be chosen so that uniform stability fails.

The classes of networks in (a)--(b)
share an important structural property in their drift, that is,
the matrix $B_1$ in \cref{Eb} is diagonal.
We establish a necessary and sufficient condition for uniform stability,
via the so-called spare capacity (safety staffing) parameter defined
in \cref{Evarrho}.
For the networks under consideration,
we show  that if the spare capacity is negative, then the limiting diffusion is 
transient under any stationary Markov control, if it is zero,
the diffusion cannot be positive recurrent, and if it is positive,
the diffusion limit is uniformly stable over all stationary Markov controls
(see \cref{Ttran,thm-main}).
The analogous results for the diffusion-scaled processes are also established. 
Lastly, we provide a characterization of the spare capacity 
parameter for the limiting diffusion when the latter is positive recurrent.
We show in \cref{GP} that the spare capacity is equal to an average `idleness'
weighted by the critical quantity in \cref{E-tran}. 

To prove the uniform exponential ergodicity for the limiting controlled diffusion, 
we use  a common Lyapunov function given in \cref{ED4.1C}. 
This Lyapunov function consists
of two components that
treat the positive and negative half spaces of the state space in a 
delicate manner. 
An important `tilting' parameter must be carefully chosen to 
account for not only the different effects of queueing and idleness (positive and 
negative half state space), but also the second order derivatives of the
extended generator of the diffusion. 
Note that these Lyapunov functions
differ from the quadratic Lyapunov functions used in 
\cite{AP16, AP18, AP19, DG13, APS19}  for the study
of stability under either constant controls or assuming abandonment, and also differ from that used in  \cite{ AHP18} for the uniform stability of the `V'~network. In \cite{DG13} for example, the stability analysis requires the existence of a common quadratic Lyapunov function which cannot be shown for the models under consideration. As it will be clear to the reader later in the paper, the uniform stability analysis without abandonment requires the use of the sum of two functions where each of them `dominates' the other over a part of the state space. See for example the proofs of \cref{Lstable,L4.1}.

The same Lyapunov function is used to prove the uniform exponential ergodicity 
for the prelimit diffusion-scaled processes.
However, unlike the `V'~network studied in \cite{AHP18},
the Foster-Lyapunov equations for the limiting diffusion do not carry over to
the analogous equations for the diffusion-scaled queueing processes over
the entire state space.
The reason lies in the jointly work conserving (JWC) condition 
(that is, all the queues have to be empty when there are idle servers) which is
essential in establishing the weak convergence to the controlled limiting diffusion
(see \cite{Atar-05b,Atar-05a}).
To tackle this difficulty, we first provide an explicit `drift' representation
of the diffusion-scaled processes which differs from the drift of
the diffusion by an extra term that accounts
for the deviation from the JWC condition in the $n^{\rm th}$ system,
and which vanishes in the limit.
A natural extension of the concept of work conservation for multiclass multi-pool
networks is minimization of the idle servers at all times. 
This defines an action space which we call 
\emph{system-wide work conserving} (SWC).
Establishing the ``uniform" geometric ergodicity over all SWC Markov policies when
the spare capacity is positive, is accomplished by first proving a useful
upper bound for the minimum of idle servers and cumulative queue size
for the $n^{\mathrm th}$ system, and then using this to derive the Foster--Lyapunov drift
inequalities in the region of the state space where the drifts of the
diffusion limit and the $n^{\mathrm th}$ system do not match. 
This facilitates establishing the drift inequalities for the
diffusion-scaled processes.
As a consequence of the Foster-Lyapunov equations,
the invariant probability measures of the diffusion-scaled queueing processes
have uniform exponential tails. 

The property of interchange of limits attests to the validity of the diffusion
approximation for the queueing network.
For stochastic networks in the conventional heavy traffic regime, we refer the readers 
to the papers \cite{gamarnik-zeevi,Budhiraja-09,Gurvich-16,YY-16,YY-18,BDM-17}
and references therein.
For the `V'~network in the Halfin--Whitt regime, interchange of limits is established in
\cite{GS12,AHP18}.
For the `N'~network, Stolyar \cite{Stolyar-15b} has shown the interchange of limits
under a specific static priority policy. 
This property also holds for networks with pool-dependent service rates under the
LQFS-LB scheduling policy, as shown in \cite[Section 7.2]{Stolyar-Yudovina-13}. 
Stolyar and Yudovina \cite{Stolyar-Yudovina-12} and Stolyar \cite{Stolyar-15b} then 
proved tightness of the stationary distributions and interchange of limits
 of a leaf-activity priority policy in the sub-diffusion and diffusion scales, 
 respectively, in the underloaded regime. 
This paper contributes to this literature by establishing that the limit of the
diffusion-scaled invariant distributions is equal to the invariant distribution of
the limiting diffusion process for the large classes of networks considered
under any stationary Markov policy (see Remark~\ref{R5.3}).  

\subsection{Organization of the paper}
In the next subsection, we summarize the notation used in the paper.
In \cref{S2.1}, we describe the model and state informally the assumptions used.
We define the diffusion scaled processes, and characterize the corresponding
controlled generator in \cref{S2.2}.
In \cref{S2.3}, the notion of \emph{system-wide work conserving policies} is
introduced, and this is used in \cref{S2.4} to take limits and establish the
diffusion approximation.
In \cref{S3}, we define the parameter of \emph{spare capacity} ($\varrho$)
for multiclass multi-pool networks and show that whenever $\varrho<0$,
the process is transient under any stationary
Markov control both for the diffusion limit and the $n^{\mathrm{th}}$ system for the models under consideration.
In the same subsection, we establish the relation between the spare capacity and
average idleness.
In \cref{S4} we first provide equivalent characterizations of uniform exponential
ergodicity of controlled diffusions, and then proceed to
establish that the diffusion limits of
the aforementioned classes of networks are uniformly exponentially ergodic
and their invariant probability measures have uniform exponential tails.
Finally, \cref{S5} is devoted to the study of uniform exponential ergodicity
of the $n^{\mathrm{th}}$ system of networks under consideration.

\subsection{Notation} 
We use $\mathbb{R}^m$ (and $\mathbb{R}^m_+$), $m\ge 1$,
to denote real-valued $m$-dimensional (nonnegative) vectors, and write $\RR$
for the real line.
We use $z\transp$ to denote the transpose of a vector $z\in\Rm$.
Throughout the paper $e\in\Rm$ stands for the vector whose elements are
equal to $1$, that is, $e=(1,\dotsc,1)\transp$, and $e_i\in\Rm$ denotes
the vector whose elements are all $0$ except for the $i^{\mathrm{th}}$ element
which is equal to $1$.
For $x, y\in \RR$,
$x \vee y = \max\{x,y\}$, $x\wedge y = \min\{x,y\}$, 
$x^+ = \max\{x, 0\}$ and $x^- = \max\{-x,0\}$.

For a set $A\subseteq\RR^m$, we use $A^{\mathsf c}$, $\partial A$, and $\Ind_{A}$
to denote the complement, the boundary, and the indicator function of $A$, respectively.
A ball of radius $r>0$ in $\RR^m$ around a point $x$ is denoted by $\sB_{r}(x)$,
or simply as $\sB_{r}$ if $x=0$.
We also let $\sB \equiv \sB_{1}$.
The Euclidean norm on $\RR^m$ is denoted by $\abs{\,\cdot\,}$,
and $\langle \cdot\,,\,\cdot\rangle$ stands for the inner product.
For $x\in\RR^m$, we let $\norm{x}^{}_1\df \sum_i \abs{x_i}$,
and by $K_r$, or $K(r)$, for $r>0$, we denote the closed cube
\begin{equation}\label{Ecube}
K_r \,\df\, \{x\in\RR^m \colon \norm{x}^{}_{1} \le r\}\,.
\end{equation}
Also, we define $x\mx\df \max_i\, x_i$, and $x\mn\df \min_i\, x_i$,
and $x^\pm\df \bigl(x_1^\pm,\dotsc,x_m^\pm\bigr)$.

For a finite signed measure $\nu$ on $\Rm$,
and a Borel measurable $f\colon\Rm\to[1,\infty)$,
the $f$-norm of $\nu$ is defined by
\begin{equation}\label{Efnorm}
\norm{\nu}_f \,\df\, \sup_{\substack{g\in\cB(\Rm), \; \abs{g}\le f}}\;
\babss{\int_{\Rm} g(x)\,\nu(\D{x})}\,,
\end{equation}
where $\cB(\Rm)$ denotes the class of Borel measurable functions on $\Rm$.

\section{The queueing network model and the diffusion limit}

In this section, we consider a sequence of parallel server networks whose processes,
parameters, and variables are indexed by $n$. We recall
some of the definitions and notations used in \cite{AP16,AP19}.

\subsection{Model and assumptions}\label{S2.1}

Consider a general Markovian parallel server (multiclass multi-pool) network with $m$ classes of
customers and $J$ server pools. Customer classes take values
in $\cI=\{1,\dotsc,m\}$ and server pools  in $\cJ=\{1,\dotsc,\nJ\}$.
Forming their own queue, customers of each class are served according to
a First-Come-First-Served (FCFS) service discipline.
We assume throughout the paper that customers do not abandon.
For all $i\in\cI$, let $\cJ(i)$ denote the subset of server pools that can serve
customer class $i$. On the other hand, for all $j\in\cJ$, let $\cI(j)$
be the subset of customer classes that can be served by server pool $j$.

We form a bipartite undirected graph $\cG = (\cI\cup \cJ, \cE)$ with a set of edges defined by
$\cE = \{(i,j)\in\cI\times\cJ\colon j\in\cJ(i)\}$, and use the notation $i \sim j$,
if $(i,j)\in\cE$, and $i \nsim j$, otherwise.
We assume that the graph $\cG$ is a tree. 
We define
\begin{equation}\label{eqn-R-G}
\RR^{\cG}_+ \,\df\, \bigl\{\xi=[\xi_{ij}]\in\RR^{m \times \nJ}_+\,\colon
\xi_{ij}=0~~\text{for~}i\nsim j\bigr\}\,,
\end{equation}
and analogously define $\RR^{\cG}$, $\ZZ^{\cG}_+$, and $\ZZ^{\cG}$. 

In each server pool $j$, we let $N_j^n$ be the number of servers, and assume that the servers are statistically identical.
For each $i\in\cI$, class $i$ customer arrives according to a Poisson process
with arrival rate $\lambda_i^n > 0$. 
These customers are served at an exponential rate $\mu_{ij}^n >0$ at server pool $j$ 
if $j\in\cJ(i)$, and $\mu_{ij}^n=0$ otherwise.
Finally, we assume that the arrival and service processes of all classes
are mutually independent.
We study these networks in the Halfin--Whitt regime, which involves the
following assumption on the parameters.
There exist positive constants $\lambda_i$ and  $\nu_j$,
nonnegative constants  $\mu_{ij}$, with $\mu_{ij}>0$
for $i\sim j$ and $\mu_{ij}=0$ for $i\nsim j$,  and constants 
$\Hat{\lambda}_i$, $\Hat{\mu}_{ij}$ and $\Hat{\nu}_j$, such that
the following limits exist as $n\to\infty$:
\begin{equation}\label{EHW}
\frac{\lambda^{n}_{i} - n \lambda_{i}}{\sqrt{n}} \;\to\;\Hat{\lambda}_{i}\,,\qquad
{\sqrt{n}}\,(\mu^{n}_{ij} - \mu_{ij}) \;\to\;\Hat{\mu}_{ij}\,,\quad\text{and}\quad
 \frac{N^{n}_{j} -  n\nu_{j}}{\sqrt{n}} \;\to\;  \Hat{\nu}_j\,.
\end{equation}
The parameters $\lambda_i$ and $\mu_{ij}$ are the limiting arrival and service rates, $\nu_j$ is the limiting service capacity in pool $j$ in the fluid scale, while the parameters $\Hat{\lambda}_i$, $\Hat{\mu}_{ij}$ and $\Hat{\nu}_j$ are the associated limits in the diffusion scale.

An additional standard assumption
referred to as the \emph{complete resource pooling} condition
\cite{williams-2000, Atar-05b} concerns the fluid scale equilibrium,
and is stated as follows.
The linear program (LP) given by
\begin{equation}\label{ELP}
\text{Minimize} \quad  \max_{j \in \cJ}\;\sum_{i \in \cI(j)} \xi_{ij}\,,
\quad \text{subject to} \quad  \sum_{j \in \cJ(i)} \mu_{ij} \nu_j \xi_{ij}
\,=\, \lambda_i \ \ \forall\, i \in \cI\,, 
\end{equation}
has a unique solution
$\xi^*=[\xi^*_{ij}]\in\RR^{\cG}_+$  satisfying 
\begin{equation} \label{ELP-id}
\sum_{i \in \cI} \xi^*_{ij} \,=\, 1, \quad \forall j \in \cJ \,,
\quad\text{and}\quad \xi^*_{ij}>0\quad \text{for all~} i \sim j\,.
\end{equation}
We define $x^*\in\RR^m$, and $z^* \in\RR^{\cG}_+$ by
\begin{equation}\label{Efluid}
x_i^* = \sum_{j\in\cJ} \xi^*_{ij}\nu_j\,,
\quad\text{and\ \ } z^*_{ij} = \xi^*_{ij}\nu_j\,.
\end{equation}
The quantity $\xi^*_{ij}$ represents the fraction of servers in pool $j$ allocated to class $i$ in the fluid equilibrium, $x^*_i$ represents the total number of class $i$ customers in the system, and $z^*_{ij}$ is the number of class $i$ customers in pool $j$. 
Note that the constraint in \eqref{ELP} is the rate balance equation for each class $i$ with allocations in each service pool $j$. Also, $\rho_j:=\sum_i\xi_{ij}$ can be interpreted as the traffic intensity in pool $j$, hence the condition in \eqref{ELP-id} implies that each pool is critically loaded.

For each $i\in\cI$ and $j\in\cJ$, we let $X_i^n=\{X_i^n(t)\colon t\ge 0\}$
denote the total number of class $i$ customers in the system
(both in service and in queue), $Z_{ij}^n=\{Z_{ij}^n(t),\,t\ge0 \}$ the number of
class $i$ customers currently being served in pool $j$,
$Q_i^n=\{Q_i^n(t),\,t\ge0\}$ the number of class $i$ customers in the queue,
and $Y_j^n=\{Y_j^n(t),\,t\ge0\}$ the number of idle servers in server pool $j$.
Let $X^{n} = (X_i^{n})_{i \in \cI}$, $Y^{n} = (Y_j^{n})_{j \in \cJ}$, 
$Q^{n} = (Q_i^{n})_{i \in \cI}$,
and $Z^{n} = (Z_{ij}^{n})_{i \in \cI,\, j \in \cJ}$.
The process $Z^{n}$ is the scheduling control. 
We have clearly the following \emph{balance equations}
\begin{equation}\label{Eqy2}
\begin{aligned}
Q^n_i(t) &\,\df\, X_i^n(t) - \sum_{j\in\cJ} Z^n_{ij}(t)\,,\quad i\in\cI\,,\\[5pt]
Y_j^n(t) &\,\df\,   N_j^n - \sum_{i\in\cJ}Z^n_{ij}(t)\,,
\quad j\in\cJ\,,
\end{aligned}
\end{equation}
Dropping the explicit dependence on $n$ for simplicity,
let $(x,z)\in\ZZ_+^m\times \ZZ^{\cG}_+$ denote a state-action pair. We rewrite \cref{Eqy2} as
\begin{equation}\label{Eqy}
\begin{aligned}
q_i(x, z) &\,\df\, x_i - \sum_{j\in\cJ} z_{ij}\,,\quad i\in\cI\,,\\[5pt]
y_j(z) &\,\df\,   N_j^n - \sum_{i\in\cJ}z_{ij}\,,
\quad j\in\cJ\,.
\end{aligned}
\end{equation}
and define the \emph{action space} $\cZn(x)$ by
\begin{equation*}
\cZn(x)\,\df\, \bigl\{z \in \ZZ^{\cG}_+\,\colon
q_i(x,z) \wedge y_j^n(z) =0\,,~
q_i(x,z)\ge0\,,~y_j^n(z) \ge0 \quad\forall\,(i,j)\in\cE\bigr\}\,. 
\end{equation*}
Note that this space consists of work-conserving actions only. It should be noted here that there is an abuse of notation in \cref{Eqy}. The quantities $q_i(x,z)$ and $y_j(z)$ still represent the number of class $i$ customers in the queue and the number of idle servers in pool $j$ respectively. Equation \cref{Eqy} is used to show the dependence on $x$ and $z$ through the balance equations.

\subsection{Diffusion scaling}\label{S2.2}

With $\xi^*\in\RR^\cG_+$ the solution of the (LP), we
define the centering quantities of the diffusion-scaled processes $\Bar{z}^n \in\RR^{\cG}_+$ and $\Bar{x}^n\in\RR^m$ by
\begin{equation}\label{Eequil}
 \Bar{z}^n_{ij} \,\df\, \frac{1}{n} \xi^*_{ij}N_j^n\,,
 \qquad \Bar{x}^n_i \,\df\, \sum_{j\in\cJ} \Bar{z}^n_{ij}\,,
\end{equation}
and
\begin{equation} \label{Edszx}
\begin{aligned}
\Hat{X}^{n}_i(t) &\,\df\,
\frac{1}{\sqrt{n}} \bigl(X_i^{n}(t) - n \Bar{x}^n_{i}\bigr) \,,  \\[5pt]
\Hat{Q}^{n}_i(t) &\,\df\, \frac{1}{\sqrt{n}} Q_i^{n}(t) \,,
\end{aligned}
\qquad
\begin{aligned}
\Hat{Z}^{n}_{ij}(t) &\,\df\, \frac{1}{\sqrt{n}}
\bigl(Z_{ij}^{n}(t) - n \Bar{z}^n_{ij}\bigr)\,,\\[5pt]
\Hat{Y}^{n}_j(t) &\,\df\, \frac{1}{\sqrt{n}} Y_j^{n}(t) \,.
\end{aligned}
\end{equation}

Using \cref{Eqy2,Edszx}, these obey the \emph{centered balance equations}
\begin{equation}\label{Ebal}
\begin{aligned}
& \Hat{X}^{n}_i(t) \,=\, \Hat{Q}_i^{n}(t) + \sum_{j \in \cJ(i)} \Hat{Z}^{n}_{ij}(t)
\qquad \forall\, i \in \cI\,, \\[5pt]
 & \Hat{Y}^{n}_j(t) + \sum_{i\in \cI(j)} \Hat{Z}^{n}_{ij}(t) \,=\, 0
 \qquad \forall\, j \in \cJ \,.
\end{aligned}
\end{equation}
We introduce suitable notation in the diffusion scale as follows
(see \cite[Definition~2.3]{AP19}).

For $x\in\ZZ^{m}_{+}$ and $z\in \cZn(x)$, we define
\begin{equation}\label{ED2.1A}
\Hat{x}^{n}\,\df\, \frac{x - n \Bar{x}^n}{\sqrt{n}}\,, \qquad
 \Hat{z}^n \,\df\, \frac{z - n \Bar{z}^n}{\sqrt{n}}\,,
\end{equation}
and let $\sS^n$ denote the state space in the diffusion scale, that is,
\begin{equation}\label{EsSn}
\sS^n \,\df\,
\bigl\{\Hat{x}\in\Rm\,\colon \sqrt{n}\Hat{x}+n\Bar{x}^n\in\ZZ^m_{+}\bigr\}\,.
\end{equation}
It is clear that the diffusion-scaled work-conserving action space $\hcZn(\Hat{x})$
takes the form
\begin{equation*}
\hcZn(\Hat{x}) \,\df\,
\bigl\{\Hat{z}\,\colon
\sqrt{n}\Hat{z}+n \Bar{z}^n\in\cZn(\sqrt{n}\Hat{x}+n \Bar{x}^n)\bigr\}\,,
\qquad \Hat{x}\in\sS^n\,.
\end{equation*}

Recall that a scheduling policy is called stationary Markov if $Z^n(t) = z(X^n(t))$
for some function $z\colon\ZZ^{m}_+ \to \ZZ^{\cG}_+$, in which case we identify the
policy with the function $z$.
Under a stationary Markov policy, $X^n$ is Markov with controlled generator 
\begin{equation}\label{EcL}
\cL_z^n f(x) \,\df\, \sum_{i\in\cI}\Biggl(\lambda_i^n\bigl(f(x+e_i)-f(x)\bigr)
+ \sum_{j\in\cJ(i)}\mu_{ij}^n z_{ij}\bigl(f(x-e_i)-f(x)\bigr)\Biggr)
\end{equation} 
for $f\in \Cc(\RR^m)$ and $x \in \ZZ^{m}_{+}$.
Let $\ell^{n} = (\ell^{n}_1,\dotsc,\ell^{n}_m)\transp$ be defined by
\begin{equation}\label{Eelln}
\ell^{n}_i\,\df\, \frac{1}{\sqrt{n}} \biggl(\lambda^{n}_i
-  \sum_{j \in \cJ(i)} \mu_{ij}^{n} \xi^*_{ij}N_j^n \biggr) \,.
\end{equation}
By \cref{Eequil}, the assumptions on the parameters in
\cref{EHW,ELP}, we have
\begin{equation*}
\ell^{n}_i \;\xrightarrow[n\to\infty]{}\;
\ell_i\,\df\, \Hat{\lambda}_i - \sum_{j \in \cJ(i)} \Hat{\mu}_{ij} z_{ij}^* -\sum_{j\in\cJ(i)} \mu_{ij} \xi^*_{ij} \hat{\nu}_j\,,
\end{equation*}
with $z^*$ as in \cref{Efluid}.
Let $\ell\df(\ell_1,\dotsc,\ell_m)\transp$. Note that $\ell_i^n$ and $\ell_i$ can be regarded as the deficit or surplus in the number of servers (of order $\order(\sqrt{n}$) allocated to class $i$ in the diffusion scale. Note also that $\ell_i^n$ and $\ell_i$ appears as constants in the drift of the diffusion-scaled and diffusion limit processes; see \cref{Ebn,Eb}

We drop the dependence on $n$ in the diffusion-scaled
variables in order to simplify the notation.
A work-conserving stationary Markov policy $z$,
that is a map $z\colon\ZZ^{m}_+ \to \ZZ^{\cG}_+$
such that $z(x)\in\cZn(x)$ for all $x\in\ZZ_+^m$, gives rise to a policy
$\Hat{z}\colon\sS^n\to\RR^\cG$, with $\Hat{z}(\Hat{x}) \in\hcZn(\Hat{x})$
for all $\Hat{x}\in\sS^n$, via \cref{ED2.1A} (and vice-versa).
Using \cref{EcL,Edszx,Eelln} and rearranging terms, the controlled generator
of the corresponding
diffusion-scaled process can be written as
\begin{equation}\label{E-Gn}
\begin{aligned}
\widehat\cL_{\Hat{z}}^n f(\Hat{x}) &\,=\,
\sum_{i\in\cI}\frac{\lambda_i^n}{n}\,
\frac{\dd f\bigl(\Hat{x};\tfrac{1}{\sqrt n} e_i\bigr)
+ \dd f\bigl(\Hat{x};-\tfrac{1}{\sqrt n} e_i\bigr)}{n^{-1}} \\
&\mspace{200mu}-\sum_{i\in\cI} b^n_i(\Hat{x},\Hat{z})\,
\frac{\dd f\bigl(\Hat{x}; -\tfrac{1}{\sqrt n} e_i\bigr)}{n^{-\nicefrac{1}{2}}}\,,
\quad \Hat{x}\in\sS^n\,,\ \Hat{z}\in\hcZn(\Hat{x})\,,
\end{aligned}
\end{equation}
where
$\dd f$ is given by
\begin{equation*}
\dd f(x;y) \,\df\, f(x+y) -f(x)\,,\quad x,y\in\Rm\,,
\end{equation*}
and the `drift' $b^n=(b_1^n,\cdots,b_m^n)\transp$ is given by
\begin{equation}\label{E-GnC}
b^n_i(\Hat{x},\Hat{z}) \,\df\, \ell_i^n - \sum_{j\in\cJ(i)}\mu_{ij}^n
\Hat{z}_{ij}\,,\quad\Hat{z}\in \hcZn(\Hat{x})\,,\ i\in\cI\,.
\end{equation}

Abusing the notation for $\Hat{x}\in\sS^n$ and $\Hat{z}\in\hcZn(\Hat{x})$, we define
(compare with \cref{Ebal})
\begin{equation}\label{ED2.2A}
\Hat{q}^n_i(\Hat{x},\Hat{z})\,\df\, \Hat{x}_i - \sum_{j\in\cJ(i)} \Hat{z}_{ij}\,,
\quad i\in\cI\,,
\qquad
\Hat{y}^n_j(\Hat{z})\,\df\, - \sum_{i\in\cI(j)} \Hat{z}_{ij}\,, \quad j\in\cJ\,,
\end{equation}
and
\begin{equation}\label{ED2.2B}
\Hat{\vartheta}^n(\Hat{x},\Hat{z})\,\df\,
\langle e, \Hat{q}^n(\Hat{x},\Hat{z})\bigr\rangle\wedge\,
\langle e, \Hat{y}^n(\Hat{z})\bigr\rangle\,.
\end{equation}
Recall \cref{Edszx}. The parameter $\hat{\vartheta}^n$ can therefore be regarded as the scaled minimum of the total number of customers in the queues and the total number of idle servers.

By \cref{ED2.2A}, we have
\begin{equation}\label{ES2.2A}
\bigl\langle e, \Hat{q}^n(\Hat{x},\Hat{z})\bigr\rangle
\,=\,\Hat{\vartheta}^n(\Hat{x},\Hat{z})
+ \langle e, \Hat{x}\rangle^+\,,
\quad\text{and}\quad
\bigl\langle e,\Hat{y}^n(\Hat{z})\bigr\rangle
\,=\,\Hat{\vartheta}^n(\Hat{x},\Hat{z})
+ \langle e, \Hat{x}\rangle^-
\end{equation}
for all $\Hat{x}\in\sS^n$ and $\Hat{z}\in\hcZn(\Hat{x})$.
Define the $(m-1)$ and $(J-1)$ simplexes
\begin{equation}\label{E-simp}
\varDelta_c \,\df\, \{u\in\RR^m \,\colon u\ge0\,,\ \langle e,u\rangle = 1\}\,,
\quad\text{and\ \ }
\varDelta_s \,\df\, \{u\in\RR^{\nJ} \,\colon u\ge0\,,\ \langle e,u\rangle = 1\}\,,
\end{equation}
and let $\varDelta\df \varDelta_c\times\varDelta_s$.
By \cref{ES2.2A}, there exists $u=(u^c,u^s)\in\varDelta$
such that
\begin{equation}\label{ES2.2B}
\Hat{q}^n(\Hat{x},\Hat{z}) \,=\,\bigl(\Hat{\vartheta}^n(\Hat{x},\Hat{z})
+\langle e, \Hat{x}\rangle^+\bigr)\,u^c\,,
\quad\text{and}\quad
\Hat{y}^{n}(\Hat{z}) \,=\, \bigl(\Hat{\vartheta}^n(\Hat{x},\Hat{z})
+\langle e, \Hat{x}\rangle^-\bigr)\,u^s\,.
\end{equation}
Let
\begin{equation*}
\cD \,\df\, \Bigl\{ (\alpha,\beta) \in \Rm\times \RR^{J}\,\colon
\textstyle\sum_{i=1}^m \alpha_i = \sum_{j=1}^J \beta_j\Bigr\}\,.
\end{equation*}
As shown in \cite[Proposition~A.2]{Atar-05a},
there exists a unique linear map $\Phi=[\Phi_{ij}]\colon \cD\to \RR^{\cG}$
solving
\begin{equation} \label{EPhi}
\sum_{j\in\cJ(i)} \Phi_{ij}(\alpha,\beta)\,=\,\alpha_i \quad \forall\, i \in \cI\,,
\quad\text{and}\quad
\sum_{i\in\cI(j)} \Phi_{ij}(\alpha,\beta)\,=\,\beta_j \quad \forall\, j \in \cJ\,.
\end{equation}
The solutions $\Phi_{ij}$ correspond to the resource allocations $\hat{z}_{ij}$ in the diffusion scale, see \eqref{ES2.2C}.
Since $\bigl(\Hat{x}-\Hat{q}^{n}(\Hat{x},\Hat{z}),- \Hat{y}^{n}(\Hat{z})\bigr)\in\cD$
by \cref{Eqy,ED2.2A},
using the linearity of the map $\Phi$ and \cref{ES2.2B,EPhi}, it follows that
\begin{equation}\label{ES2.2C}
\begin{aligned}
\Hat{z} &\,=\,
\Phi\bigl(\Hat{x}-\Hat{q}^{n}(\Hat{x},\Hat{z}),- \Hat{y}^{n}(\Hat{z})\bigr)\\
&\,=\, \Phi\bigl(\Hat{x}-\langle e,\Hat{x}\rangle^{+} u^c,
-\langle e,\Hat{x}\rangle^{-} u^s\bigr)
-\Hat{\vartheta}^{n}(\Hat{x},\Hat{z})\,\Phi(u^c,u^s)\,.
\end{aligned}
\end{equation}

We describe an important property of the linear map $\Phi$ which
we need later.
Consider the matrices $B_1^n\in\RR^{m\times m}$ and
$B_2^n\in\RR^{m\times J}$ defined by
\begin{equation}\label{Elinear}
\sum_{j\in\cJ(i)} \mu^n_{ij} \Phi_{ij}(\alpha,\beta) \,=\,
\bigl(B_1^n \alpha + B_2^n \beta\bigr)_i\,,\quad
\forall\,i\in\cI\,,\ \forall (\alpha,\beta)\in \cD\,.
\end{equation}
It is clear that for $B_1^n$ to be a nonsingular matrix the
basis used in the representation of the linear map $\Phi$ should be
of the form $\mathfrak{D} =\bigl(\alpha, (\beta)_{-j}\bigr)$, $j\in\cJ$, where
$(\beta)_{-j}= \{\beta_\ell\,,\, \ell\ne j\}$.
Since $\Phi$ has a unique representation in terms of such a basis, and since
$B_i^n$, $i=1,2$, are determined uniquely from $\Phi$ by \cref{Elinear},
abusing the terminology, we refer to such an $\mathfrak{D}$ as a basis for
$B_i^n$, $i=1,2$.
In \cite[Lemma~4.3]{AP16}, the following property is asserted:
Given any $\Hat\imath\in\cI$, there exists an ordering of $\{\alpha_i\,,\;i\in\cI\}$
with $\alpha_{\Hat\imath}$ the last element, and $\Hat\jmath\in\cJ$,
such that the matrix
$B_1^n$ is lower diagonal with positive diagonal elements with
respect to this ordered basis $\bigl(\alpha, (\beta)_{-\Hat\jmath}\bigr)$.
For more details, we refer the reader to \cite[Section~4.1]{AP16}.

In view of \cref{ES2.2C,Elinear},
for any $\Hat{z}\in\tcZn(\Hat{x})$ with $\Hat{x}\in\sS^n$,
there exists $u=u(\Hat{x},\Hat{z})\in\varDelta$ such that the drift
$b^n$ in \cref{E-GnC} takes the form
\begin{equation}\label{Ebn}
b^n(\Hat{x},\Hat{z}) \,=\, \ell^n 
- B_1^n \bigl(\Hat{x}-\langle e,\Hat{x}\rangle^{+} u^c\bigr)
+ B_2^n u^s \langle e,\Hat{x}\rangle^{-}
+\Hat{\vartheta}^{n}(\Hat{x},\Hat{z}) \bigl(B_1^n u^c + B_2^n u^s\bigr)\,.
\end{equation}

\subsection{Joint and system-wide work conservation}\label{S2.3}

We start with the following definition.

\begin{definition}\label{D2.3}
We say that an action $\Hat{z}\in\hcZn(\Hat{x})$ is \emph{jointly work
conserving} (JWC), if $\Hat{\vartheta}^n(\Hat{x},\Hat{z})=0$. Recall that a work conserving policy refers to an action in which a server is idle if and only if there is no customer waiting in the queue that this server can serve. A jointly work conserving action keeps all servers busy unless
all queues are empty.
Let
\begin{equation*}
\Hat{\vartheta}^n_\ast(\Hat{x}) \,\df\, \min_{\Hat{z}\in\hcZn(\Hat{x})}\,
\Hat{\vartheta}^n(\Hat{x},\Hat{z})\,,\quad \Hat{x}\in\sS^n\,,
\end{equation*}
and
\begin{equation*}
\tcZn(\Hat{x}) \,\df\,
\bigl\{\Hat{z}\in\hcZn(\Hat{x})\,\colon \Hat{\vartheta}^n(\Hat{x},\Hat{z})
= \Hat{\vartheta}^n_\ast(\Hat{x})\bigr\}\,,\quad \Hat{x}\in\sS^n\,.
\end{equation*}
We refer to $\tcZn(\Hat{x})$ as the
\emph{system-wide  work conserving} (SWC) action set at $\Hat{x}$.
A stationary Markov scheduling policy $\Hat{z}$ is called SWC if
$\Hat{z}(\Hat{x})\in\tcZn(\Hat{x})$ for all $\Hat{x}\in\sS^n$.
We let $\tfZn$ denote the class of all such policies.
Since $z$ and $\Hat{z}$ are related by \cref{ED2.1A}, abusing this terminology,
we also refer to a Markov policy $z\colon\ZZ_+^m\to\ZZ_+^\cG$
as SWC, if it satisfies
\begin{equation*}
\frac{z(x) - n \Bar{z}^n}{\sqrt{n}}\,\in\,
\tcZn\biggl(\frac{x - n \Bar{x}^n}{\sqrt{n}}\biggr)\,,
\end{equation*}
and we write $z\in\tfZn$.
\end{definition}
 
We recall \cite[Lemma~3]{Atar-05b} which states that there exists
$M_0>0$ such that the collection
of sets $\Breve{\sX}^{n}$ defined by
\begin{equation}\label{EhsX}
\Breve{\sX}^{n}\,\df\,
\bigl\{\Hat{x} \in \sS^n\,\colon  \norm{\Hat{x}}^{}_1\le M_{0}\, \sqrt{n}  \bigr\}\,,
\end{equation} 
has the following property.
If $\Hat{x}\in\Breve{\sX}^{n}$, then for any pair $(\Hat{q},\Hat{y})$
such that $\sqrt{n}\Hat{q}\in\ZZ^m_+$, $\sqrt{n}\Hat{y}\in\ZZ^J_+$, and satisfying
\begin{equation*}
\langle e,\Hat{q}\rangle\wedge \langle e,\Hat{y}\rangle \,=\,0\,,\quad
\langle e,\Hat{x}-\Hat{q}\rangle \,=\, \langle e,-\Hat{y}\rangle\,,
\quad\text{and\ \ }
\Hat{y}_j\le N^n_j\,,\quad j\in\cJ\,,
\end{equation*}
 it holds that
$\Phi(\Hat{x}-\Hat{q},-\Hat{y})\in\hcZn(\Hat{x})$.
It follows from this lemma and \cref{D2.3} that
if $\Hat{x}\in\Breve{\sX}^{n}$, then the actions in $\tcZn(\Hat{x})$
are JWC.

\begin{remark}
Using \cref{ES2.2A}, we know that under the JWC condition
\begin{equation*}
\bigl\langle e, \Hat{q}^n(\Hat{x},\Hat{z})\bigr\rangle
\,=\,\langle e, \Hat{x}\rangle^+\,,
\quad\text{and}\quad
\bigl\langle e,\Hat{y}^n(\Hat{z})\bigr\rangle
\,=\,\langle e, \Hat{x}\rangle^-\,.
\end{equation*}
Rewriting \cref{ES2.2B} under the JWC condition, we therefore have the following
\begin{equation*}
u^c_i \,=\,
\begin{cases}
\frac{\hat{q}^n_i(\hat x, \hat z)}{\langle e, \Hat{q}^n(\Hat{x},\Hat{z})\bigr\rangle} \,,\qquad \text{if } \langle e, \Hat{q}^n(\Hat{x},\Hat{z})\bigr\rangle > 0 \,,\\
e_1\,, \qquad \text{otherwise}
\end{cases}
\end{equation*}
and 
\begin{equation*}
u^s_j \,=\,
\begin{cases}
\frac{\hat{y}^n_j(\hat z)}{\langle e, \Hat{y}^n(\Hat{z})\bigr\rangle} \,,\qquad \text{if } \langle e, \Hat{y}^n(\Hat{z})\bigr\rangle > 0 \,,\\
e_1\,, \qquad \text{otherwise}.
\end{cases}
\end{equation*}
Hence, one can see that the control $u^c_i$ represents the proportion of the total queue length in the network at queue $i$, while $u^s_j$ represents the proportion of the total number of idle servers in the network at pool $j$.

Recall from \cref{D2.3} that a JWC action keeps all servers busy unless
all queues are empty. It is clear that this cannot be always enforced in multiclass multi-pool
networks over the entire state space. A SWC control enforces the complementarity between overall queue length and service.

\end{remark}

\subsection{The diffusion limit}\label{S2.4}
The \emph{diffusion approximation} or \emph{diffusion limit} of the
queueing model described above is an
$m$-dimensional stochastic differential equation (SDE) of the form 
(see \cite{Atar-05a} and \cite[Section 2.5]{Atar-05b})
\begin{equation}\label{E-sde}
\D X_t\,=\,b(X_t,U_t)\,\D{t} +\upsigma(X_t)\, \D W_t\,,\qquad X_0=x\in \mathbb{R}^m\,.
\end{equation}
Here, $\process{W}$ is a standard $m$-dimensional Brownian motion,
and the control $U_t$ takes values in the set $\varDelta= \varDelta_c\times\varDelta_s$
defined in \cref{E-simp}.
The drift $b$ can be derived as follows.
Recall $\RR^{\cG}$ in \eqref{eqn-R-G}.
For $u=(u^c,u^s)\in\varDelta$, let $\widehat{\Phi}[u]\colon\Rm\to\RR^{\cG}$
be defined by
\begin{equation}\label{EhPhi}
\widehat{\Phi}[u](x) \,\df\,
\Phi\bigl(x- (e\cdot x)^{+} u^c, - (e\cdot x)^{-} u^s\bigr)\,,
\end{equation}
with $\Phi$ as defined in \cref{EPhi}.
Then the drift $b$ takes the form
\begin{equation}\label{Edrift}
b_i(x,u)\,=\, \ell_i - \sum_{j\in\cJ(i)}\mu_{ij}\widehat{\Phi}_{ij}[u](x)\,.
\end{equation}
By \cite[Lemma~4.3]{AP16}, we also know that \cref{Edrift} can be expressed as
\begin{equation}\label{Eb}
b(x,u) \,=\, \ell - B_1\bigl(x- \langle e,x \rangle^+ u^c\bigr)
+ B_2 u^s \langle e,x \rangle^-\,,
\end{equation}
where $B_1\in\RR^{m\times m}$ is a lower diagonal matrix with positive
diagonal elements, and $B_2\in\RR^{m\times J}$.
Of course $B_i$ in \cref{Eb} and $B_i^n$ in \cref{Ebn}, $i=1,2$,
have the same functional form with respect
to $\{\mu_{ij}\}$ and  $\{\mu^n_{ij}\}$, respectively.

The diffusion matrix $\upsigma\in\RR^{m\times m}$ is constant, and
\begin{equation*}
a \,\df\, \upsigma\upsigma\transp
\,=\,\diag(2\lambda_1,\dots,2\lambda_m)\,.
\end{equation*}
In addition, for $f\in \Cc^2(\RR^m)$, we define
\begin{equation}\label{ELg}
\Lg_u f(x) \,\df\, \frac{1}{2} \trace\bigl(a\nabla^2f(x)\bigr)
+\bigl\langle b(x,u),\nabla f(x)\bigr\rangle\,,
\end{equation}
with $\nabla^2f$ denoting the Hessian of $f$.

In the following, we describe the drift of the diffusion limit of the networks under consideration while showing some examples in \cref{{fig-networks}}.
\subsubsection{Networks with a dominant server pool}\label{S2.4.1}

This network has one non-leaf server node, which, without loss of generality,
we label as $j=1$.
As in \cref{S2.1}, the customer nodes are denoted by $\cI = \{1,2,\dotsc,m\}$, and
the server nodes by $\cJ = \{1,2,\dotsc,\nJ\}$.
Recall that $\cJ(i)$ is the collection of sever nodes connected to customer $i$.
Owing to the tree structure of the network,
server $1\in\cJ(i)$ for all $i\in\{1,2,\dotsc,m\}$.
Let $\cJ_1(i)\df\cJ(i)\setminus\{1\}$ for all $i\in\cI$.
Recall the form of the drift in \cref{Edrift}.
Using \cref{EPhi}, it is simple to show that the matrix $\widehat{\Phi}_{ij}[u]$
for this network is given by
\begin{equation}\label{EhPhiM}
\widehat{\Phi}_{ij}[u](x) \,=\, 
\begin{cases}
x_i -  \langle e,x \rangle^+u_i^c + \sum_{j\in\cJ_1(i)}
 \langle e,x \rangle^-\,u_j^s & \text{for\ } j=1\,,\\[5pt]
- \langle e,x \rangle^- \,u_j^s& \text{for\ } j\in\cJ_1(i)\,,\\[5pt]
0 &\text{otherwise.}
\end{cases}
\end{equation}    
Using \cref{EhPhiM}, the drift takes the following simple form:
\begin{equation} \label{EdriftM}
b_i(x,u)\,=\,\ell_i - \mu_{i1}\bigl(x_i - u_i^c \langle e,x\rangle^+\bigr)
+ \sum_{j\in\cJ_1(i)}\mu_{i1}\bigl(\eta_{ij} - 1\bigr)u_{j}^s \langle e,x \rangle^-\,,
\qquad i\in\cI\,,
\end{equation}
with $\eta_{ij} \df \frac{\mu_{ij}}{\mu_{i1}}$ for $j\in\cJ_1(i)$ and $i\in\cI$.
Note that $B_1=\diag(\mu_{11},\dotsc,\mu_{m1})$,
and so $\ell = -\frac{\varrho}{m} B_1e$,
where $\varrho$ is given by \cref{Evarrho}.
We define
\begin{equation*}
\Bar\eta \,\df\,  \max_{i\in\cI}\,\max_{j\in\cJ_1(i)}\, \eta_{ij}\,,
\quad\text{and\ \ }
\underline\eta \,\df\, \min_{i\in\cI}\,\min_{j\in\cJ_1(i)}\, \eta_{ij}\,.
\end{equation*}

\begin{figure*}[htp]
\centering
\begin{subfigure}[t]{0.325\textwidth}
\centering
    \includegraphics[height=1.1in]{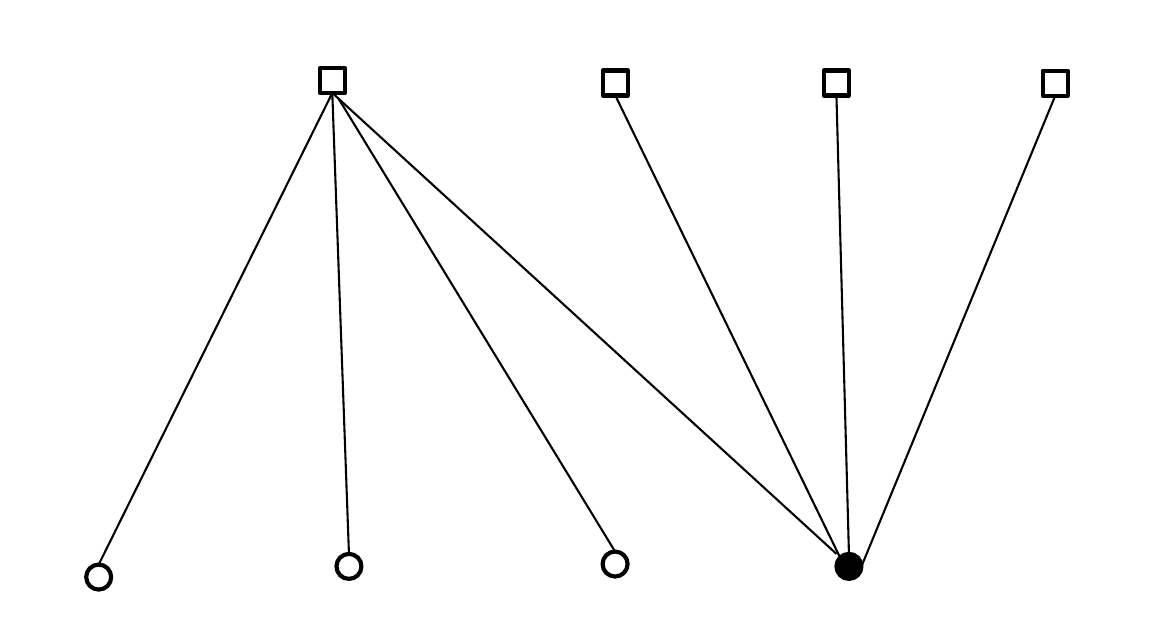} 
    \caption{Generalized `N' Network}
    \end{subfigure}
    \begin{subfigure}[t]{0.33\textwidth}
    \centering
  \includegraphics[height=1.1in]{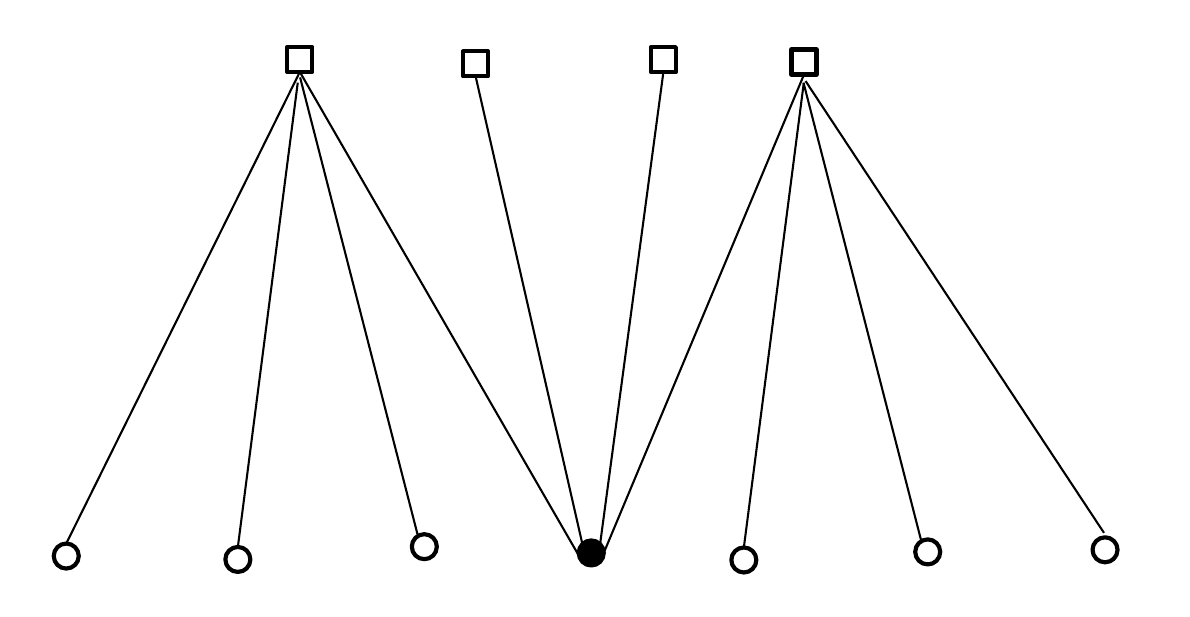}
    \caption{Generalized `M' Network}
  \end{subfigure}
  \begin{subfigure}[t]{0.33\textwidth}
  \centering
   \includegraphics[height=1.1in]{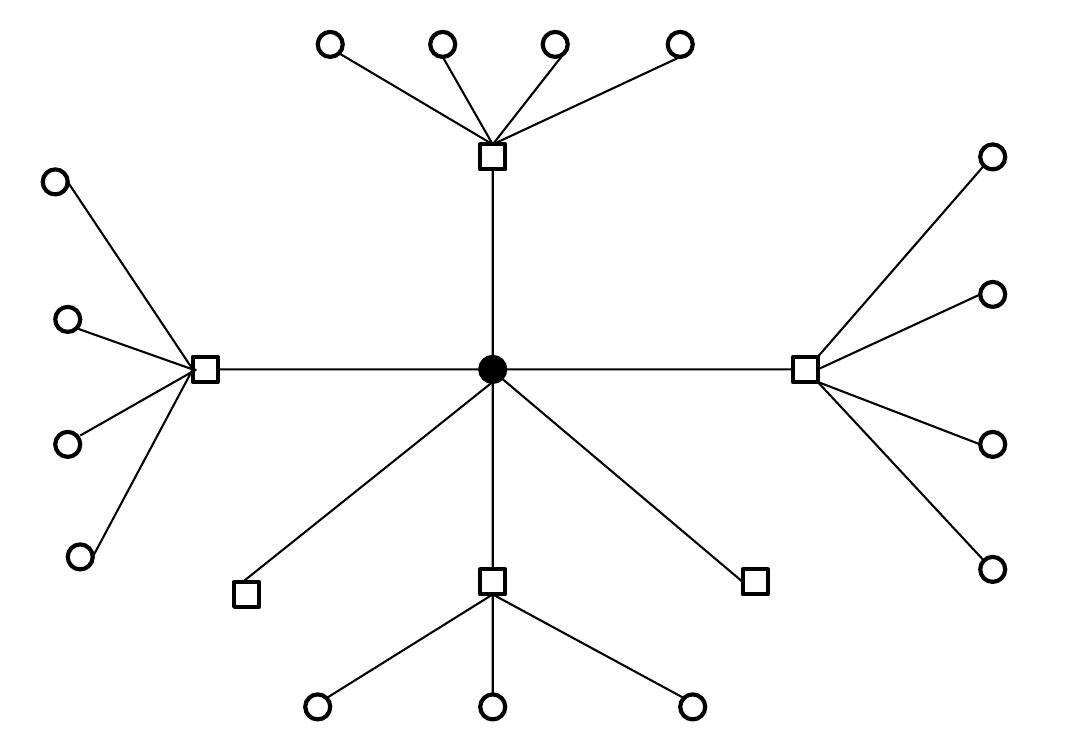}
    \caption{Network with a dominant server pool}
   \end{subfigure}
\caption{Examples of multiclass multi-pool networks with a dominant pool
(square--customer classes, circle--server pools, solid circle--the
dominant server pool)
}
 \label{fig-networks}
\end{figure*}

\subsubsection{Networks with class-dependent service rates}\label{S2.4.2}
We consider in this part arbitrary tree networks where the service rates are dictated
by the customer type; namely $\mu_{ij}\,=\,\mu_i$ for all $(i,j)\in\cE$.
Recall the definition in \cref{EhPhi}.
Using \cref{EPhi,Edrift}, the drift of this network takes the form
\begin{equation}\label{Edrift-i}
b_i(x,u) \,=\, \ell_i - \sum_{j\in\cJ(i)}\mu_{ij}\widehat{\Phi}_{ij}[u](x)
\,=\, \ell_i - \mu_i\bigl(x_i - u_i^c\langle e,x \rangle^+\bigr)\,,
\qquad \forall i\in\cI\,.
\end{equation}
Note then that $B_1 = \diag(\mu_1, \dots, \mu_m)$ and $B_2=0$.

We  remark that both classes of networks have one common feature: the matrix $B_1$ is diagonal. 

\subsection{A necessary and sufficient condition for uniform stability} 

We define the spare capacity (or the safety staffing) for the
$n^{\mathrm{th}}$ system (prelimit)
and the diffusion limit by 
\begin{equation}\label{Evarrho}
\varrho_n \,\df\, -\,\bigl\langle  e(B_1^n)^{-1},\ell^n\,\bigr\rangle\,,
\quad\text{and\ \ }
\varrho \,\df\, -\,\bigl\langle  eB_1^{-1},\ell\,\bigr\rangle\,,
\end{equation}
respectively.
Note, of course, that $\varrho_n\to\varrho$ as $n\to\infty$ by \cref{EHW}. Recall the expressions of $\ell^n$ and $\ell$ in \cref{Eelln}. Recall that $\ell_i^n$ and $\ell_i$ can be regarded as the deficit or surplus in the number of servers (of order $\order(\sqrt{n}$) allocated to class $i$ in the diffusion scale.  Hence $\varrho_n$ and $\varrho$ can be regarded as the optimal 
reallocation of the capacity
fluctuations (positive or negative) of order $\sqrt{n}$ 
when each server pool employs 
a square-root staffing rule.

We summarize the main results of the paper. 

\begin{theorem} \label{thm-main}
Consider a network with a dominant server pool, or with class-dependent service rates.
Then the conditions  $\varrho>0$ and $\varrho_n >0$ are necessary and sufficient
for the uniform stability of the limiting diffusion and the diffusion-scaled
queueing processes, respectively. More precisely:
\begin{itemize}
\item[(i)]
if  $\varrho<0$, the process $\process{X}$ in \cref{E-sde} is transient under any
stationary Markov control.
In addition, if $\varrho = 0$, then $\process{X}$ cannot be positive recurrent.

\item[(ii)]
if $\varrho_n <0$, the process $\process{X^n}$ is transient under any
stationary Markov scheduling policy. 
In addition, if $\varrho_n = 0$, then $\process{X^n}$ cannot be positive recurrent.

\item[(iii)] if $\varrho >0$, the processes  $\process{X}$  are uniformly
exponentially ergodic over stationary Markov controls.

\item[(iv)] if $\varrho_n >0$, the processes $\process{X^n}$ are uniformly
exponentially ergodic over SWC scheduling policies, and the invariant distributions
have exponential tails.
\end{itemize}
\end{theorem}

Parts (i) and (ii) of \cref{thm-main} follow from
\cref{Ttran,Ctran}, respectively.
Part (iii) follows from \cref{Tstable}, and part (iv) from \cref{T5.1}.

\section{Two properties of the spare capacity and transience} \label{S3}
In the first part of this section we prove the results in
\cref{thm-main} (i) and (ii).
It is important to note that for the models in \cref{S2.4.1,S2.4.2} we have
\begin{equation}\label{E-tran}
1+ \langle e,B_1^{-1}B_2 u^s \rangle \,>\,0\,.
\end{equation}
Note that for networks with a dominant server pool, we have 
$1 = u_1^s + \sum_{i \in \cI} \sum_{j \in 
\cJ_1(i)} u_j^s$. Hence
\begin{equation*}
\begin{aligned}
1 + \langle e, B_1^{-1} B_2 u^s \rangle & = 1 + \sum_{i \in \cI} \sum_{j \in 
\cJ_1(i)} (\eta_{ij} - 1)u_j^s\\
& = 1 - \sum_{i \in \cI} \sum_{j \in 
\cJ_1(i)} u_j^s + \sum_{i \in \cI} \sum_{j \in 
\cJ_1(i)} \eta_{ij}u_j^s  \\
&= u_1^s + \sum_{i \in \cI} \sum_{j \in 
\cJ_1(i)} \eta_{ij}u_j^s > 0.  
\end{aligned}
\end{equation*}
Note also that for networks with a class dependent service rate we have established that $B_2 = 0$. Hence $1 + \langle e, B_1^{-1} B_2 u^s \rangle = 1 > 0$.

\begin{proposition}\label{Ttran}
Suppose that $\varrho= - \langle e,B_1^{-1}\ell\rangle<0$.
Then the process $\process{X}$ in \cref{E-sde} is transient under any
stationary Markov control.
In addition, if $\varrho = 0$, then $\process{X}$
cannot be positive recurrent.
\end{proposition}

\begin{proof}
Recall that the first and second order derivatives of the hyperbolic tangent function are
\begin{equation*}
\tanh'(x) = \frac{1}{\cosh^2(x)}\,; \qquad \tanh''(x) = -2 \frac{\tanh(x)}{\cosh^2(x)}
\end{equation*}
Let	$H(x) \df \tanh\bigl(\beta\langle e, B_1^{-1} x\rangle\bigr)$, with $\beta>0$.
Then
\begin{equation*}
\trace\bigl(a\nabla^2 H(x))\bigr) \,=\,
\beta^2\tanh''\bigl(\beta\langle e, B_1^{-1} x\rangle\bigr)
\babs{\upsigma\transp B_1^{-1}e}^2\,.
\end{equation*}
We have
\begin{equation}\label{PTtranA}
\begin{aligned}
\Lg_u H(x) &\,=\, \frac{1}{2}\trace\bigl(a\nabla^2 H(x)\bigr)
+ \bigl\langle b(x,u),\nabla H(x)\bigr\rangle \\
& \,=\, - \beta^2\frac{\tanh\bigl(\beta\langle e, B_1^{-1} x\rangle\bigr)}
{\cosh^2\bigl(\beta\langle e, B_1^{-1} x\rangle\bigr)}
\abs{\upsigma\transp B_1^{-1} e}^2\\
&\mspace{80mu}+ \frac{\beta}{\cosh^2\bigl(\beta\langle e, B_1^{-1} x
\rangle\bigr)}\biggl(\bigl\langle e, B_1^{-1}\ell\bigr\rangle 
+ \langle e,x\rangle^-
\Bigl(1 + \bigl\langle e , B_1^{-1}B_2 u^s\bigr\rangle\Bigr)\biggr)\,.
\end{aligned}
\end{equation}
Thus, for
$0<\beta < \langle e, B_1^{-1}\ell\rangle\, \abs{\upsigma\transp B_1^{-1} e}^{-2}$,  
we obtain $\Lg_u H(x) >0$ by  \cref{E-tran}.
Therefore, $\{H\bigl(X_t\bigr)\}_{t\ge0}$ is a bounded submartingale,
so it converges almost surely.
Since $X$ is irreducible, it can be either recurrent or transient. 
If it is recurrent, then $H$
should be constant a.e.\ in $\RR^m$, which is not the case.
Thus $X$ is transient.

We now turn to the case where $\varrho = 0$.
Suppose that the process $\{X(t)\}_{t\ge0}$ (under some stationary
Markov control) has an invariant probability measure $\uppi(\D{x})$.
It is well known that $\uppi$ must have a positive density.
Let $h_1(x)$ and $h_2(x)$ denote respectively the first and the second terms on
the right hand side of \cref{PTtranA}.
Applying It\^{o}'s formula to \cref{PTtranA} with $X_0 = x$ as in \cref{E-sde}, we obtain
\begin{equation}\label{PTtranB}
\Exp^\uppi\bigl[H(X_{t\wedge \tau_r})\bigr] - H(x) \,=\,
\sum_{i=1,2}\Exp^\uppi\Biggl[\int_0^{t\wedge \tau_r} h_i(X_s)\D s\Biggr]\,,
\end{equation}
where $\tau_r$ denotes the first exit time from $\sB_r$, $r>0$.
Note that $h_1(x)$ is bounded and $h_2(x)$ is non-negative.
Thus using dominated and monotone convergence, we can take limits in \cref{PTtranB}
as $r\to \infty$ for the terms on the right side to obtain
\begin{equation*}
\int_{\Rm} H(x)\uppi(\D{x}) - H(x) \,=\,
t \sum_{i=1,2} \int_{\Rm} h_i(x)\uppi(\D{x}), \qquad t\ge0\,.
\end{equation*}
Since $H(x)$ is bounded, we can divide both sides by $t$ and $\beta$ and take the
limit as $t\to\infty$ to get
\begin{equation}\label{PTtranC}
\int_{\Rm} \beta^{-1}h_1(x) \uppi(\D{x})
+ \int_{\Rm} \beta^{-1}h_2(x) \uppi(\D{x}) \,=\, 0\,.
\end{equation}
Since $\beta^{-1} h_1(x)$ tends to 0 uniformly in $x$ as $\beta\searrow 0$,
the first term on the left hand side of \cref{PTtranC} vanishes as
$\beta\searrow0$.
However, since $\beta^{-1}h_2(x)$ is bounded away from $0$ on the open set
$\{x\in\Rm \colon \langle e,x \rangle ^- >1\}$,
this contradicts the fact that $\uppi(\D{x})$ has full support.
\end{proof}

\begin{remark}
In the proof of \cref{Ttran}, the function $H(x)$ is a bounded test function which was chosen so that $\Lg_u H(x) \ge 0$. Assume that $X_t$ is recurrent.
Since $H(X_t)$ converges a.s. (being a bounded submartingale), and
since $X_t$ `visits every open neighborhood' in $\Rm$ with probability $1$,
it follows that $H$ must be a constant function which is a contradiction.

The formalism behind the above argument is as follows:
Let $\uptau$ be the first hitting time to the unit ball $B$ centered at $x=0$.
If $\Exp_x$ denotes the expectation operator on the canonical space of the Markov process
$\{X_t\}_{t\ge0}$ with initial condition $X_0= x$,
then by Dynkin's formula we obtain $\Exp_x[H(X_\uptau)] \ge H(x)$.
If the process is recurrent, then of course $\Prob_x(\uptau<\infty)=1$,
which gives
$\Exp_x[H(X_\uptau)]\le \sup_{y\in B}\, H(y)$.
Thus $\sup_{y\in B}\, H(y)\ge H(x)$ for all $x\in\Rm$ which is not true.
Moving the center of the the ball $B$ to an arbitrary point $z$, and denoting
it as $B(z)$, we similarly
have
$$\sup_{y\in B(z)}\, H(y)\ge H(x)\qquad\forall x,z\in\Rm\,.$$
This implies that $H(x)=$ constant which is a contradiction and hence $X$ is transient.

\end{remark}


\begin{proposition}\label{Ctran}
Suppose that $\varrho_n <0$.
Then the state process $\process{X^n}$ of the $n^{\mathrm{th}}$
system is transient under any
stationary Markov scheduling policy. In addition, if $\varrho^n = 0$, then $\process{X^n}$ cannot be positive recurrent. 
\end{proposition}

\begin{proof}
The proof mimics that of \cref{Ttran}.
We apply the function $H$ in that proof to the operator
$\widehat\cL_{\Hat{z}}^n$ in \cref{E-Gn}, and use the identity
\begin{equation}\label{Eident}
H\Bigl(x\pm \tfrac{1}{\sqrt n} e_i\Bigr) - H(x) \mp
\frac{1}{\sqrt n}\partial_{x_i} H(x)
\,=\, \frac{1}{n}\int_0^1 (1-t)\,
\partial_{x_ix_i} H\Bigl(x\pm \tfrac{t}{\sqrt n} e_i\Bigr) \,\D{t}\,,
\end{equation}
to express the first and second order incremental quotients,
together with \cref{Ebn} which implies that
\begin{multline*}
\bigl\langle b^n(\Hat{x},\Hat{z}),\nabla H(\Hat{x})\bigr\rangle
\,=\,
\frac{\beta}{\cosh^2\bigl(\beta\langle e, (B_1^n)^{-1} \hat{x}
\rangle\bigr)}\biggl(\bigl\langle e, (B_1^n)^{-1}\ell^n\bigr\rangle \\
+ \bigl(\Hat{\vartheta}^{n}(\Hat{x},\Hat{z})+\langle e,\hat{x}\rangle^-\bigr)
\Bigl(1 + \bigl\langle e , (B_1^n)^{-1}B_2^n u^s\bigr\rangle\Bigr)\biggr)\,.
\end{multline*}
The rest follows exactly as in the proof of \cref{Ttran}.
%
\end{proof}

\subsection{Spare capacity and average idleness}

It is shown in \cite{AHP18,APS19} that if
the diffusion limit of the `V' network with no abandonment has a invariant
distribution $\uppi$ under some stationary Markov control, then
$\varrho$ represents the `average idleness' of the system, that is,
$\varrho = \int_{\Rm} \langle e,x \rangle ^- \uppi(\D{x})$.
In calculating this average
for multi-pool networks, idle servers are not weighted equally across
different pools and the term $\bigl\langle e,B_1^{-1}B_2u^s(x)\bigr\rangle$
appears in the expression, see \cref{EGP-A}.
It is important to note that only the control on the idleness allocations among
server pools $u^s$ appears in the identity, and the control
component $u^c$ does not.

\begin{theorem}\label{GP}
Consider a network with a dominant server pool, or with class-dependent service rates,
and suppose that $\varrho>0$.
Let $\uppi_u$ denote the invariant invariant probability measure
corresponding to a stationary Markov control $u\in\Usm$, whose existence follows
from \cref{thm-main}\,\ttup{iii}.
Then
\begin{equation}\label{EGP-A}
\varrho \,=\,
\int_{\Rm} \Bigl(1 + \bigl\langle e,B_1^{-1}B_2u^s(x)\bigr\rangle\Bigr)
\langle e,x \rangle ^-\, \uppi_u(\D{x}).
\end{equation} 
\end{theorem}

\begin{proof}
The proof is similar to that of \cite[Corollary~5.1]{APS19},
but more involved for the multiclass multi-pool networks. 
We first recall some definitions and notations.
Let $\chi_r(t)$, $\breve{\chi}_r(t)$, $r>1$ be smooth, concave and convex functions,
respectively, defined by
\begin{equation*}
\chi_r(t)\,=\,
\begin{cases}
t\,, &t\le\,r-1\,,\\
r-\frac{1}{2}\,, & t\ge r\,,
\end{cases}
\qquad\text{and}\quad
\breve{\chi}_r(t)\,=\,
\begin{cases}
t\,, &t\ge 1-r\,,\\
\frac{1}{2}-r\,, & t\le -r\,.
\end{cases}
\end{equation*}
Let $g_r(x)=\breve{\chi}_r\bigl(\langle e, B_1^{-1}x\rangle\bigr)$,
and $f_r(x) = \chi_r\bigl(g_r(x)\bigr)$.
A straightforward calculation shows that 
\begin{align*}
\langle b(x,u),\nabla f_r(x)\rangle &\,=\, h_1(x) + h_2(x)\,,\\
\frac{1}{2}\trace\bigl(a(x) \nabla^2 f_r(x)\bigr) &\,=\, h_3(x) + h_4(x)\,,
\end{align*}
where
\begin{align*}
h_1(x)&\,\df\, -\varrho\chi'_r\bigl(f(x)\bigr)\,
\breve{\chi}'_r\bigl(\langle e, B_1^{-1} x \rangle\bigr)\,,\\
h_2(x)&\,\df\,
\bigl[ 1 + \bigl\langle e,B_1^{-1}B_2 u^s \bigr\rangle \bigl]
\chi'_r\bigl(f(x)\bigr)\,\breve{\chi}'_r\bigl(\langle e, B_1^{-1} x \rangle\bigr)
\langle e,x \rangle^-\,,\\
h_3(x)&\,\df\, \frac{1}{2} \chi_r''\bigl(f(x)\bigr)\,
\bigl(\breve{\chi}_r'\bigl(\langle e, B_1^{-1} x \rangle\bigr)\bigr)^2
\babs{\upsigma\transp B_1^{-1}e}^2\,,\\
h_4(x)&\,\,\df\, \frac{1}{2}\chi_r'\bigl(f(x)\bigr)
\breve{\chi}_r''\bigl(\langle e, B_1^{-1} x \rangle\bigr)
\babs{\upsigma\transp B_1^{-1}e}^2\,.
\end{align*}
We note that $g_r(x)$ is positive and bounded below away from $0$, and
$f_r(x)$ is smooth, bounded, and has bounded derivatives.
Also note that $h_i$, $i=1,2,3$, are bounded, and $h_2$ is nonnegative.
Therefore, if $\{X(t)\}_{t\ge 0}$ is positive recurrent with an invariant probability
measure $\uppi_u(\D{x})$, a straightforward
application of It\^o's formula shows
that $\uppi_u\bigl(\Lg_u f_r\bigr)=0$. 
Therefore, we obtain
\begin{equation}\label{PGP-A}
\uppi_u(-h_1)\,=\, \uppi_u(h_2) + \uppi_u(h_3) + \uppi_u(h_4).
\end{equation}
By the definition of $\chi_r$ and $\breve{\chi}_r$,
it is straightforward to verify that 
\begin{equation}\label{PGP-B}
\lim_{r\to\infty} \uppi_u(h_3) \,=\, \lim_{r\to\infty} \uppi_u(h_4) \,=\, 0\,.
\end{equation}
In addition, using dominated convergence theorem
\begin{equation}\label{PGP-C}
\begin{aligned}
&\lim_{r\to\infty} \uppi_u(h_1) \,=\, -\varrho\,,\\
&\lim_{r\to\infty} \uppi_u(h_2) \,=\,
\int_\Rm\bigl( 1 + \bigl\langle e,B_1^{-1}B_2 u^s \bigr\rangle \bigl)
\langle e,x \rangle^-\,\uppi_u(\D{x})\,.
\end{aligned}
\end{equation}
Combining \cref{PGP-A,PGP-B,PGP-C}, we obtain \cref{EGP-A}.
\end{proof}

\begin{remark}
Using \cref{GP} and the drift in \cref{S2.4.1,S2.4.2}, it is easy
to verify that in the case of
a network with a dominant server pool we have
\begin{equation*}
\varrho = \int_{\Rm} \Biggl[1 + \sum_{i\in\cI} \sum_{j\in\cJ_1(i)}
\bigl(\tfrac{\mu_{ij}}{\mu_{i1}} - 1\bigr)u_{j}^s \Biggr]
\langle e,x \rangle ^- \uppi_u(\D{x})\,,
\end{equation*}
whereas in the case of a network
with class-dependent service rates, \cref{EGP-A} takes the form
\begin{equation*}
\varrho \,=\,
\int_{\Rm} \langle e,x \rangle ^- \uppi_u(\D{x}).
\end{equation*}
\end{remark}

\section{Uniform exponential ergodicity of the diffusion limit}\label{S4}
 In this section we show that if $\varrho>0$ then the
diffusion limit of a network with a dominant server pool, or with
class-dependent service rates,
is uniformly exponentially ergodic and the invariant distributions
have exponential tails.

We start by reviewing the notion of \emph{uniform exponential ergodicity} for
a controlled diffusion.
We do this under fairly general hypotheses.
Consider a
controlled diffusion process $X = \{X_{t},\,t\ge0\}$
which takes values in the $m$-dimensional Euclidean space $\RR^{m}$, and
is governed by the It\^o  equation
\begin{equation}\label{E-sde0}
\D{X}_{t} \,=\,b\bigl(X_{t},v(X_t)\bigr)\,\D{t} + \upsigma(X_{t})\,\D{W}_{t}\,.
\end{equation}
All random processes in \cref{E-sde0} live in a complete
probability space $(\Omega,\sF,\Prob)$.
The process $W$ is a $m$-dimensional standard Wiener process independent
of the initial condition $X_{0}$.
The function $v$ maps $\RR^m$ to a compact, metrizable set $\Act$ and is 
Borel measurable. The collection of such functions comprising of the
set of stationary Markov controls is denoted by $\Usm$.

The parameters  of the equation \eqref{E-sde0} satisfy the following:
\begin{itemize}
\item[(1)]
\emph{Local Lipschitz continuity:\/}
The functions
$b\colon\RR^{m}\times\Act\to\RR^{m}$ and 
$\upsigma\colon\RR^{m}\to\RR^{m\times m}$
are continuous, and satisfy
\begin{equation*}
\abs{b(x,u)-b(y, u)} + \norm{\upsigma(x) - \upsigma(y)}
\,\le\,C_{R}\,\abs{x-y}\qquad\forall\,x,y\in B_R\,,\ \forall\, u\in\Act\, .
\end{equation*}
for some constant $C_{R}>0$ depending on $R>0$.

\item[(2)]
\emph{Affine growth condition:\/}
For some $C_0>0$, we have
\begin{equation*}
\sup_{u\in\Act}\; \langle b(x,u),x\rangle^{+} + \norm{\upsigma(x)}^{2}\,\le\,C_0
\bigl(1 + \abs{x}^{2}\bigr) \qquad \forall\, x\in\RR^{m}\,.
\end{equation*}
\item[(3)]
\emph{Nondegeneracy:\/}
For each $R>0$, it holds that
\begin{equation*}
\sum_{i,j=1}^{m} a^{ij}(x)\xi_{i}\xi_{j}
\,\ge\,C^{-1}_{R} \abs{\xi}^{2} \qquad\forall\, x\in B_{R}\,,
\end{equation*}
and for all $\xi=(\xi_{1},\dotsc,\xi_{m})\transp\in\RR^{m}$,
where $a= \frac{1}{2}\upsigma \upsigma\transp$.
\end{itemize}

It is well known that, under hypotheses (1)--(2),
\cref{E-sde0} has a unique strong solution 
which is also a strong Markov process for any $v\in\Usm$ \cite{Gyongy-96}.
We let $\Exp^v_x$ denote the expectation operator on the canonical space of
the process controlled by $v$, with initial condition $X_0=x$.
Let $\uptau(A)$ denote the first exit time from the set $A\in\Rm$.

We say that the process $\{X_t\}_{t\ge0}$ is
\emph{uniformly exponentially ergodic} if
for some ball $\sB_\circ$ there exist $\delta_\circ>0$ 
and $x_\circ\in \Bar{\sB}_\circ^{\mathsf c}$ such that
$\sup_{v\in\Usm}\,\Exp_{x_\circ}^v[\E^{\delta_\circ\, \uptau(\sB^{\mathsf c}_\circ)}]<\infty$.

We let
$\widehat\cA$ denote the operator
\begin{equation*}
\widehat\cA \phi(x) \,\df\, \frac{1}{2}\trace\left(a(x)\nabla^{2}\phi(x)\right)
+ \max_{u\in\Act}\, \bigl\langle b(x,u),\nabla \phi(x)\bigr\rangle\,,\qquad x\in\Rm\,,
\end{equation*}
for $\phi\in\Cc^2(\Rm)$.
For a locally bounded, Borel measurable function
$f\colon\Rm\to\RR$, which is bounded from below in
$\Rm$, i.e., $\inf_\Rm f>-\infty$,
we define the generalized principal eigenvalue of $\widehat\cA+f$ by 
\begin{equation*}
\Lambda(f)\,\df\,\inf\,\Bigl\{\lambda\in\RR\,
\colon \exists\, \varphi\in\Sobl^{2,m}(\Rm),
\, \varphi>0, \, \widehat\cA\varphi + (f-\lambda)\varphi\le 0 \text{\ a.e. in\ } \Rm
\Bigr\}\,,
\end{equation*}
where $\Sobl^{2,m}$ is a local Sobolev space.
We have the following equivalent characterizations of
uniform exponential ergodicity.
This is a straightforward extension of \cite[Theorem~3.1]{APS19}
for controlled diffusions, and is stated without proof.
Recall that a map $f\colon\Rm\to\RR$ is called \emph{coercive}, or \emph{inf-compact},
if $\inf_{\sB_r^{\mathsf c}}\,f\to\infty$ as $r\to\infty$.

\begin{theorem}\label{T4.1}
The following are equivalent.
\begin{enumerate}
\item[{\upshape(}a\/{\upshape)}]
For some ball $\sB_\circ$ there exists $\delta_\circ>0$ 
and $x_\circ\in \Bar{\sB}_\circ^{\mathsf c}$ such
that $\sup_{v\in\Usm}\,\Exp_{x_\circ}[\E^{\delta_\circ\, \uptau(\sB^{\mathsf c}_\circ)}]<\infty$.

\item[{\upshape(}b\/{\upshape)}]
For every ball $\sB$ there exists $\delta>0$ such that
$\sup_{v\in\Usm}\,\Exp_{x}^v[\E^{\delta\, \uptau(\sB^{\mathsf c})}]<\infty$ for all $x\in\sB^{\mathsf c}$.

\item[{\upshape(}c\/{\upshape)}]
For every ball $\sB$, there exists a coercive function $\Lyap\in\Sobl^{2,p}(\Rm)$,
$p>d$, with $\inf_\Rm\,\Lyap\ge1$, and positive
constants $\kappa_0$ and $\delta$ such that
\begin{equation}\label{ET4.1A}
\widehat\cA\,\Lyap(x) \,\le\, \kappa_0\,\Ind_{\sB}(x) - \delta\Lyap(x)
\qquad\forall\,x\in\Rm\,.
\end{equation}

\item[{\upshape(}d\/{\upshape)}]
\Cref{E-sde0} is recurrent, and
 $\Lambda(\Ind_{\sB^{\mathsf c}})<1$ for every ball $\sB$.
\end{enumerate}
\end{theorem}

\begin{remark} \label{R4.1}
Recall \cref{Efnorm}.
Let $P^v_t(x,\D{y})$ denote the transition probability of $\process{X}$
in \cref{E-sde0}
under a control $v\in\Usm$.
It is well known that \cref{ET4.1A} implies that there exist
constants $\gamma$ and $C_\gamma$ which do not depend on the control
$v$ chosen, such that
\begin{equation}\label{ER4.1A}
\bnorm{ P^{v}_t(x,\cdot\,)-\uppi_v(\cdot)\,}_{\Lyap}\,\le\,
C_\gamma \Lyap(x)\, \E^{-\gamma t}\,,\qquad \forall\,x\in\Rm\,,\ \forall\,t\ge0\,,
\end{equation}
where $\uppi_v$ denotes the invariant probability measure of
$\process{X}$ under the control $v$. In addition, for any control $v \in \Usm$ we have
\begin{equation}
\int_\Rm \Lyap(x)\,\uppi_v(\D{x}) \,\le\, \frac{\kappa_0}{\delta}\,.
\end{equation}
In particular for $\Lyap(x)$ being an exponential function, the moment generating function of $\process{X}$ is finite.
\end{remark}

\subsection{A class of intrinsic Lyapunov functions for the queueing network model}

As seen in \cref{T4.1}, uniform exponential ergodicity is equivalent
to the Foster--Lyapunov inequality in \cref{ET4.1A}.
In establishing this property for the diffusion limit of stochastic networks,
a proper choice of a Lyapunov function is of tantamount importance.
We first describe an intrinsic class of such functions.

We fix a convex function $\psi\in\Cc^2(\RR)$ with the property that
$\psi(t)$ is constant for $t\le-1$, and $\psi(t)=t$ for $t\ge0$.
This is defined by 
\begin{equation*}
\psi(t) \df\,
\begin{cases}
-\frac{1}{2}, & t\le-1\,,\\[5pt]
(t+1)^3 -\frac{1}{2}(t+1)^4-\frac{1}{2} & t \in [-1,0]\,,\\[5pt]
t & t\ge 0\,.
\end{cases}
\end{equation*}
For $\veo>0$ we define
\begin{equation*}
\psi_\veo(t) \,\df\, \psi (\veo t)\,,
\end{equation*}
Thus $\psi_\veo(t)=\veo t$ for $t>0$.
A simple calculation also shows that $\psi''_\veo(t)\le \frac{3}{2}\veo^2$.

Suppose that $B_1 = \diag(\Tilde\mu_1,\dotsc,\Tilde\mu_m)$.
Using the function $\psi_\veo$ introduced above, we let
\begin{equation}\label{ED4.1A}
\Psi(x)\,\df\, \sum_{i\in\cI}\frac{\psi_\veo(x_i)}{\Tilde\mu_{i}}\,,
\end{equation}
with
\begin{equation}\label{ED4.1B}
\veo\,\df\, \frac{\varrho}{3m}
\Biggl(\sum_{i\in\cI}\frac{\lambda_i(3\Tilde\mu_i+2)}{\Tilde\mu_{i}^2}
\Biggr)^{-1}\,.
\end{equation}

We also define
\begin{equation}\label{ED4.1C}
V_1(x)\,\df\,\exp\bigl(\theta\Psi(-x)\bigr)\,,\quad
V_2(x)\,\df\,\exp\bigl(\Psi(x)\bigr)\,,\quad\text{and\ \ }
V(x)\,\df\, V_1(x) + V_2(x)\,,
\end{equation}
with $\theta$ a positive constant.

As a result of fixing the value of $\veo$ in \cref{ED4.1B},
$\Psi$ depends only on the parameter $\theta$.
This simplifies the statements of the results in the rest of the paper.

We review some useful properties of the function $\psi_\veo$. Note that for $\veo >0 $ we have
\begin{equation*}
\psi'_\veo(t) t \,=\,
\begin{cases}
0 & \veo t\le-1\,,\\[5pt]
\veo t (\veo t +1)^2 (-2\veo t+1)  & \veo t \in [-1,0]\,,\\[5pt]
\veo t & \veo t\ge 0\,.
\end{cases}
\end{equation*}
The minimum of $\psi'_\veo(t) t$ when $\veo t \in [-1,0]$ is $-\frac{3(39 + 55\sqrt{33})}{4096} \ge -\frac{1}{2}$ for $\veo t = - \frac{1 + \sqrt{33}}{16}$. Therefore
\begin{equation}\label{ED4.1D}
\begin{aligned}
\sum_{i\in\cI} \psi_\veo'(x_i) x_i &= 
\sum_{i \in \cI} \psi_\veo'(x_i) x_i \Ind_{\{\veo x_i \ge 0\}} + \sum_{i \in \cI} \psi_\veo'(x_i) x_i \Ind_{\{\veo x_i \in [-1,0]\}}\\
& \ge \veo \sum_{i \in \cI} x_i\Ind_{\{x_i \ge 0\}} -\frac{1}{2} \sum_{i \in \cI} \Ind_{\{\veo x_i \in [-1,0]\}} \\
& \ge \veo \norm{x^+}_1 -\frac{m}{2}\,,
\end{aligned}
\end{equation}
and similarly $-\sum_{i\in\cI} \psi'_\veo(-x_i) x_i \,\ge\, \veo \norm{x^-}^{}_1 - \frac{m}{2}$,
where $\cI\df\{1,2,\dotsc,m\}$. Note also that
\begin{equation*}
-\sum_{i\in\cI} \psi'_\veo(-x_i) x_i \,\le\, \veo \langle e,x\rangle
\,\le\, \sum_{i\in\cI} \psi_\veo'(x_i) x_i\,.
\end{equation*}

The function $V$ in \cref{ED4.1C}, scaled by the parameter $\theta$
which are suppressed in the notation,  is our choice of a Lyapunov
function when $B_1$ is a diagonal matrix. It is constructed in an intricate manner in order to capture
both the total workload (using $\Psi(x)$) on the positive half-space and the idleness (using $\Psi(-x)$) on the negative half-space. As one cannot simply use $x^+$ or $x^-$, we must construct mollified smooth functions for them. In addition, one must recognize that the effects of the workload and idleness on the system are not the same (not symmetric), so we also introduce a parameter $\theta$ in the definition of $V_1(\cdot)$.
The reader will notice the similarities in \cref{ED4.1C}
and \cite[Definition~2.2]{AHP18}.
However the function $V$ used in this paper is the sum of
the two exponential functions $V_1$ and $V_2$, whereas their
product is used in \cite[Lemma~2.1]{AHP18}.
As will be seen later, in the case of multiclass multi-pool networks, the
analysis is considerably more complex. In \cref{S4.2} for example, the proofs required $V_1 \ge V_2^2$ on a subset of the state space while requiring the opposite inequality on its complement in order to establish the Foster-Lyapunov inequality in \cref{ELstableB}. See the discussions following \cref{PLstableA,PL4.1B}. This is mainly the reason behind using the sum instead of the product when defining the function $V$.

In the following subsections, we establish the uniform exponential ergodicity
of the networks under consideration.
To help with the exposition, we study the `N'~network in detail and then proceed
to the more general networks with a dominant server pool,
and networks with class-dependent service rates.

In establishing the desired drift inequalities, we often partition the space
appropriately, and focus on the subsets of the partition.
The following cones appear quite often in the analysis.

For $\delta\in[0,1]$, we define the cones 
\begin{equation}\label{E-cone}
\begin{aligned}
\cK_{\delta}^+ &\,\df\, \bigl\{ x\in\RR^m\,\colon \langle  e,x\rangle
\ge \delta \norm{x}^{}_1\bigr\}\,,\\
\cK_{\delta}^- &\,\df\, \bigl\{ x\in\RR^m\,\colon \langle  e,x\rangle
\le - \delta \norm{x}^{}_1\bigr\}.
\end{aligned}
\end{equation}

It is clear that $\cK_0^+$ ($\cK_0^-$) corresponds to the nonnegative (nonpositive)
canonical half-space,
and $\cK_1^+$ ($\cK_1^-$) is the nonnegative (nonpositive) closed orthant.

The following identities are very useful.
\begin{equation}\label{ED4.2A}
\langle  e,x^+\rangle \,=\, \frac{1\pm\delta}{2}\,\norm{x}^{}_1\,,\qquad
\langle  e,x^-\rangle \,=\, \frac{1\mp\delta}{2}\,\norm{x}^{}_1
\qquad\text{for\ \ } x\in\partial\cK_\delta^\pm\,,\ \delta\in[0,1]\,.
\end{equation}
In addition, it is straightforward to show that
\begin{equation}\label{ED4.2B}
\sum_{i\in\cI} \psi_\veo(x_i) \,\le\, \sum_{i\in\cI} \psi_\veo(-x_i)
\quad\text{if\ \ } x\in\cK_0^-\,.
\end{equation}
Also, the following inequality is true in general for any $\cI' \subset \cI$.
\begin{equation}\label{ED4.2D}
\sum_{i\in\cI'}\psi'_\veo(x_i)x_i- \veo\sum_{i\in\cI'} x_i
\,=\, \sum_{x_i<0\,,\,i\in\cI'} \bigl(\psi'_\veo(x_i)-\veo) x_i\,\ge\,0\,.
\end{equation}

\begin{remark}\label{R4.2}
There is an important scaling of the drift which we employ.
Note that if we let $\zeta =\frac{\varrho}{m}  e+ B_1^{-1}\ell$, with $\varrho$
as in \cref{Evarrho}, then a mere translation of the origin of the form
$\Tilde X_t = X_t + \zeta$ results in a diffusion of with the same drift
as \cref{Edrift}, except that
the vector $\ell$ gets replaced by $\ell=- \frac{\varrho}{m} B_1 e$.
Therefore,  we may assume without any loss of
generality that
the drift in \cref{Eb} takes the form
\begin{equation}\label{ER4.2A}
b(x,u) \,=\, -\frac{\varrho}{m}B_1e
-B_1\bigl(x-\langle  e,x\rangle^+u^c\bigr) + B_2 u^s \langle e,x \rangle ^-.
\end{equation}
\end{remark}

\subsection{The Foster--Lyapunov inequality}

Recall the definition of the operator
$\Lg_u$ in \cref{ELg}.
We start with the following simple assertion.

\begin{lemma}\label{Lstable}
Let $V$ be the function in \cref{ED4.1C} with $\veo$ as in \cref{ED4.1B},
and $\theta \ge \theta_0$ where $\theta_0$ is a constant.
Suppose that for any $\delta\in(0,1)$ there exist positive constants
$c_0$ and $c_1$ such that the drift $b$ in \cref{ER4.2A} satisfies
\begin{subequations}
\begin{align}
\langle b(x,u),\nabla V(x) \rangle &\,\le\, 
c_0 - \veo c_1 \norm{x}^{}_1 V(x)
 \qquad\forall\,(x,u)\in\bigl(\cK_\delta^+\bigr)^{\mathsf c}\times\varDelta\,,\label{ELstableA1}
\\[5pt]
\langle b(x,u),\nabla V_2(x) \rangle &\,\le\,
- \frac{\varrho\veo}{m} V_2(x)
\qquad\mspace{58mu} \forall\,(x,u)\in\cK_\delta^+\times\varDelta\,.\label{ELstableA2}
\end{align}
\end{subequations}
Then, there exists a constant $C_0$ such that
\begin{equation}\label{ELstableB}
\Lg_u V\bigl(x\bigr)
\,\le\, C_0 - \frac{\varrho\veo}{3m} V(x)\qquad
\qquad\forall (x,u)\in\Rm\times\varDelta\,.
\end{equation}
\end{lemma}

\begin{proof}
A straightforward calculation, using the fact that
$\psi''_\veo(t)\le \frac{3}{2}\veo^2$, shows that
\begin{equation*}
\frac{1}{2}\,\trace\bigl(a\nabla^2V_2(x)\bigr)
\,\le\, \veo^2\sum_{i\in\cI} \frac{\lambda_i(3\mu_i+2)}{2\mu_i^2}
\, V_2(x)
\qquad\forall\,x\in\Rm\,.
\end{equation*}
Therefore, the choice of $\veo$ in \cref{ED4.1B} implies that
$\frac{1}{2}\trace\bigl(a\nabla^2 V_2(x)\bigr)\le\frac{\varrho\veo}{4m}V_2$
for all $x\in \mathbb{R}^2$,
and thus
\begin{equation}\label{PLstableA}
\Lg_u V_2(x) \,\le\,
- \frac{3\varrho\veo}{4m} V_2(x)
\qquad \forall\,(x,u)\in\cK_\delta^+\times\varDelta
\end{equation}
by \cref{ELstableA2}.
Since $\abs{x^+} \ge \frac{1+\delta}{1-\delta}\abs{x^-}$ for all
$x\in \cK_\delta^+$, we may select $\delta$ sufficiently close to $1$ such
that $V_2\ge V_1^2$ on $\cK_\delta^+\cap\cK_r^{\mathsf c}$ for some $r>0$.
Since $V_2$ has exponential growth in $\norm{x}^{}_{1}$ on
$\cK_\delta^+$ and
$\Lg_u V_1(x) \,\le\, C(1+\abs{x}^{}_1) V_1(x)$ on
$\cK_\delta^+\times\varDelta$, it then follows that \cref{ELstableB} holds
on $\cK_\delta^+\times\varDelta$ by \cref{PLstableA}.
It is also clear that \cref{ELstableB} also holds on
$\bigl(\cK_\delta^+\bigr)^{\mathsf c}\times\varDelta$ by \cref{ELstableA1}.
This completes the proof.
\end{proof}
We now have the following result.

\begin{theorem}\label{Tstable}
Assume that $\varrho>0$. Let $V$ be the function in \cref{ED4.1C} with $\veo$ as in \cref{ED4.1B},
and $\theta \ge \theta_0$ where $\theta_0$ is a constant.
Then the diffusion limit of any network with a dominant server pool
or with  class-dependent service rates is uniformly exponentially ergodic
and the invariant distributions have exponential tails. 
In particular, there exists $C_0$ such that 
\begin{equation}\label{ETstableA}
\Lg_u V\bigl(x\bigr)
\,\le\, C_0 - \frac{\varrho\veo}{3m} V(x)\qquad
\qquad\forall (x,u)\in\Rm\times\varDelta\,.
\end{equation}
In particular, 
\begin{equation}\label{ETstableB}
\bnorm{ P^{v}_t(x,\cdot\,)-\uppi_v(\cdot)\,}_{V}\,\le\,
C_\gamma V(x)\, \E^{-\gamma t}\,,\qquad \forall\,x\in\Rm\,,\ \forall\,t\ge0\,,
\end{equation}
where $P^v_t(x,\D{y})$ denote the transition probability of $\process{X}$
under $v = (u^c, u^s)$ and $\uppi_v$ denotes the invariant probability measure of
$\process{X}$ under the control $v$.
\end{theorem}

\begin{proof}
In \cref{L4.2,L4.3} in the section which follows,  we establish
\cref{ELstableA1,ELstableA2} for these networks.
Thus the proof of \cref{ETstableA} follows directly from \cref{Lstable}. 
\end{proof}


\subsection{Three technical lemmas}
In this section, we prove \cref{ELstableA1,ELstableA2} for the networks under consideration which implies \cref{Tstable}. Even though the `N'~network is a special case
of the networks with a dominant server pool,
we first establish the result for this network in \cref{L4.1} as understanding the results in $\RR^2$ will definitely help the reader in understanding the equations in \cref{L4.2,L4.3}.
\subsubsection{The case of the `N'~network}\label{S4.2}
Here, $m=2$, and the matrices $B_i$, $i=1,2$, in \cref{Eb} are given by
\begin{equation*}
B_1 \,=\, \begin{pmatrix} \mu_{11} & 0\\[5pt] 0 & \mu_{21}\end{pmatrix}\,,
\qquad\text{and\ \ }
B_2  \,=\, \begin{pmatrix}0& \mu_{12} - \mu_{11}\\[5pt] 0 & 0\end{pmatrix}.
\end{equation*}
Thus, using \cref{ER4.2A}, the drift
$b\colon\mathbb{R}^2\to\mathbb{R}^2$ for the `N'~network is given by
\begin{equation}\label{EdriftN}
b(x, u) \,=\,
-\frac{\varrho}{2}\begin{pmatrix} \mu_{11}\\[5pt] \mu_{21} \end{pmatrix} 
- \begin{pmatrix} \mu_{11} & 0\\[5pt] 0 & \mu_{21}\end{pmatrix}
\bigl(x-\langle  e,x\rangle^+u^c\bigr)
+ \begin{pmatrix} (\mu_{12} - \mu_{11}) u_2^s\\[5pt] 0\end{pmatrix}
\langle  e,x\rangle^-\,.
\end{equation}

Note that for the `N'~network, we have
$\Psi(x) = \frac{\psi_\veo(x_1)}{\mu_{11}} + \frac{\psi_\veo(x_2)}{\mu_{21}}$
by  \cref{ED4.1A}.
Recall the definition of the cube $K_r$ in \cref{Ecube}.
We have the following lemma that verifies the drift inequalities 
\cref{ELstableA1,ELstableA2}  for the `N' network.
\begin{lemma}\label{L4.1}
Consider an `N'~network satisfying $\varrho > 0$.
Let $\delta\in(0,1)$, $\theta\ge \theta_0\df 2  (\eta\vee\eta^{-1})$,
with $\eta \df\frac{\mu_{12}}{\mu_{11}}$,
and $V(x)$ be as in \cref{ED4.1C}.
Then, \cref{ELstableA1,ELstableA2} hold with $m=2$.
\end{lemma}

\begin{proof}
To simplify the notation we define
\begin{equation}\label{EF}
F_i(x,u) \,\df\, \frac{1}{V_i(x)}\,\bigl\langle b(x,u),\nabla V_i(x)\bigr\rangle\,,
\quad i=1,2\,.
\end{equation}
We use \cref{EdriftN}, and apply \cref{ED4.1D} and the inequalities
$\frac{\varrho}{2} \sum_{i\in\cI}\psi'_\veo \le\varrho \veo$, and
\begin{equation*}
\begin{aligned}
\psi'_\veo(-x_1)(\eta - 1)u_2^s \langle e,x \rangle &\,\le\,
-\veo (1-\eta)^+ \langle e,x \rangle\\
&\,\le\, \veo (1-\eta)^+ \norm{x^-}^{}_1
\qquad\forall\,(x,u)\in\cK_0^-\times\varDelta\,.
\end{aligned}
\end{equation*}
to obtain
\begin{equation}\label{PL4.1A}
\begin{aligned}
\frac{1}{\theta}\,F_1(x,u) &\,=\,
\frac{\varrho}{2} \sum_{i\in\cI}\psi'_\veo(-x_i)
+ \sum_{i\in\cI}\psi'_\veo(-x_i) x_i
- \psi'_\veo(-x_1)(\eta - 1)u_2^s \langle e,x\rangle^-\\
&\,\le\, 1+ \varrho \veo -\veo \norm{x^-}^{}_1 + \veo (1-\eta)^+ \norm{x^-}^{}_1\\
&\,\le\, (1+\varrho \veo) - \veo (\eta\wedge 1) \norm{x^-}^{}_1\\
&\,\le\, (1+\varrho \veo) - \frac{\veo}{2} (\eta\wedge 1) \norm{x}^{}_1
\qquad\forall\,(x,u)\in\cK_0^-\times\varDelta\,.
\end{aligned}
\end{equation}

Similarly, we have
\begin{equation}\label{PL4.1B}
\begin{aligned}
F_2(x,u) &\,=\, -\frac{\varrho}{2}\sum_{i\in\cI}\psi'_\veo(x_i)
- \sum_{i\in\cI}\psi'_\veo(x_i)x_i+\psi'_\veo(x_1)(\eta - 1)u_2^s
\langle e,x\rangle^-\\
&\,\le\, \veo\bigl(1 + (\eta - 1)^+ \bigr)\norm{x}^{}_1\\
&\,\le\, \veo(\eta \vee 1)\norm{x}^{}_1
\qquad\forall\,(x,u)\in\cK_0^-\times\varDelta\,.
\end{aligned}
\end{equation}
Note that, due to \cref{ED4.2B} and the choice of $\theta$,
we have $V_1\ge V_2^2$ on $\cK_0^-$.
Thus, since $V_1$ has exponential growth in $\norm{x}^{}_1$ on $\cK_0^-$,
combining \cref{PL4.1A,PL4.1B} and choosing an appropriate
cube $K_r$, we obtain
\begin{equation}\label{PL4.1C}
\langle b(x,u),\nabla V(x)\rangle \,\le\,
\Bigl(\theta(1+\varrho \veo) -\frac{\veo}{4} (\eta\wedge1)\norm{x}^{}_1\Bigr) V(x)
\qquad\forall\,(x,u)\in(\cK_0^-\setminus K_r)\times\varDelta\,.
\end{equation}

We continue with estimates on $\cK_0^+$.
A straightforward calculation shows that 
\begin{equation*}
\begin{aligned}
\frac{1}{\theta}\,F_1(x,u) &\,=\,
\frac{\varrho}{2}\sum_{i\in\cI}\psi'_\veo(-x_i)
+ \sum_{i\in\cI}\psi'_\veo(-x_i)x_i - \sum_{i\in\cI}
\psi'_\veo(-x_i)u_i^c\langle e,x\rangle\\
F_2(x,u) &\,=\, -\frac{\varrho}{2}\sum_{i\in\cI}\psi'_\veo(x_i)
- \sum_{i\in\cI}\psi'_\veo(x_i)x_i
+ \sum_{i\in\cI}\psi'_\veo(x_i)u_i^c\langle e,x\rangle
\end{aligned}
\qquad \forall\,(x,u)\in\cK_0^+\times\varDelta\,.
\end{equation*}
Again using \cref{ED4.1D}, we have 
\begin{equation}\label{PL4.1D}
\frac{1}{\theta}\,F_1(x,u) \,\le\,
1+\varrho \veo - \veo\norm{x^-}^{}_1 \qquad \forall\,(x,u)\in\cK_0^+\times\varDelta\,.
\end{equation}
We break the estimate of $F_2$ in two parts.
First, for any $\delta\in(0,1)$, using \cref{ED4.1D}, we obtain
\begin{equation}\label{PL4.1E}
\begin{aligned}
F_2(x,u) &\,\le\, -\frac{\varrho\veo}{2} + 1
-\veo\norm{x^+}^{}_1 + \veo \langle e,x \rangle\\
&\,\le\, -\frac{\varrho\veo}{2} + 1 -\veo\norm{x^-}^{}_1
\qquad\forall\,(x,u)\in\bigl(\cK_0^+\setminus\cK_\delta^+\bigr)\times\varDelta\,.
\end{aligned}
\end{equation}
Combining \cref{PL4.1D,PL4.1E}, we get
\begin{equation}\label{PL4.1F}
\langle b(x,u),\nabla V(x)\rangle \,\le\,
 \biggl(\theta(1+\varrho\veo) - \frac{\veo(1-\delta)}{2}\norm{x}^{}_1\biggr)V(x)
 \qquad\forall\,(x,u)\in\bigl(\cK_0^+\setminus\cK_\delta^+\bigr)\times\varDelta\,,
\end{equation}
and $\theta\ge1$,
where we use the fact that $\norm{x^-}^{}_1\ge \frac{1-\delta}{2}\norm{x}^{}_1$
on $\cK_0^+\setminus\cK_\delta^+$ by \cref{ED4.2A}.
Thus \cref{ELstableA1} follows by \cref{PL4.1C,PL4.1F}.

Next, using \cref{ED4.1D}, we have
\begin{equation*}
F_2(x,u) \,\le\,
-\frac{\varrho\veo}{2}  
\qquad\forall\,(x,u)\in\cK_\delta^+\times\varDelta\,,
\end{equation*}
and this completes the proof.
\end{proof}

\subsubsection{The case of networks with a dominant server pool}\label{S4.3}

Consider the class of networks described in
\cref{S2.4.1}.
We have the following lemma.

\begin{lemma}\label{L4.2}
Consider a network with a dominant server pool, such that $\varrho>0$.
Let $\delta\in(0,1)$,
$\theta \ge \theta_0\df 2\,\frac{\max_i{\mu_{i1}}}{\min_i{\mu_{i1}}}$,
and $V(x)$ be as in \cref{ED4.1C}.
Then, \cref{ELstableA1,ELstableA2} hold.
\end{lemma}

\begin{proof}
The method we follow is analogous to the proof of \cref{L4.1}.
Recall the definitions in \cref{EF}.
A straightforward calculation using \cref{EdriftM} shows that
\begin{equation*}
\begin{aligned}
\frac{1}{\theta}\,F_1(x,u)
&\,=\, \frac{\varrho}{m}\sum_{i\in\cI}\psi'_\veo(-x_i)
+ \sum_{i\in\cI}\psi'_\veo(-x_i)\bigl(x_i - u_i^c \langle e,x\rangle^+\bigr)
- \langle e,x \rangle^-
\sum_{i\in\cI}\sum_{j\in\cJ_1(i)}\psi'_\veo(-x_i)(\eta_{ij}-1)u_{j}^s\,,\\
F_2(x,u) &\,=\,-\frac{\varrho}{m}\sum_{i\in\cI}\psi'_\veo(x_i)
- \sum_{i\in\cI}\psi'_\veo(x_i)\bigl(x_i - u_i^c \langle e,x\rangle^+\bigr)
+ \langle e,x \rangle^-
\sum_{i\in\cI}\sum_{j\in\cJ_1(i)}\psi'_\veo(x_i)(\eta_{ij}-1)u_{j}^s\,.
\end{aligned}
\end{equation*}
Let $\underline\eta \df \min_{ij} \eta_{ij}$, and
$\Bar\eta\df \max_{ij} \eta_{ij}$.
Noting that 
\begin{equation*}
- \langle e,x \rangle^-
\sum_{i\in\cI}\, \sum_{j\in\cJ_1(i)}\psi'_\veo(-x_i)(\eta_{ij}-1)u_{j}^s
\,\le\, \veo (1-\underline\eta)^+\langle e,x \rangle^-\,,
\end{equation*}
it is easy to verify using \cref{E-cone,ED4.2A} that
\begin{equation}\label{PL4.2A}
\begin{aligned}
\frac{1}{\theta}\,F_1(x,u) &\,\le\,
\varrho \veo + \frac{m}{2} - \veo \norm{x^-}^{}_1 
+ \veo(1-\underline\eta)^+ \norm{x^-}^{}_1\\
&\,\le\,\, \varrho\veo + \frac{m}{2}
- \frac{\veo(1-\delta)}{2} (\underline\eta\wedge 1) \norm{x}^{}_1
\qquad \forall (x,u)\in\bigl(\cK_\delta^+\bigr)^{\mathsf c}\times\varDelta\,.
\end{aligned}
\end{equation}

Note that the drift equations on $\cK_0^+$ are similar to
those of the `N' model, with the only exception that
$\frac{\varrho}{2}$ is replaced by $\frac{\varrho}{m}$,
and the sum ranges from $i=1,\dotsc,m$ instead of $i=1,2$.
Hence, we obtain
\begin{equation}\label{PL4.2B}
\begin{aligned}
F_2(x,u) &\,\le\,
-\frac{\varrho\veo}{m} + \frac{m}{2} - \veo\norm{x^+}^{}_1 + \veo \langle e,x \rangle
\\
&\,\le\, -\frac{\varrho\veo}{m} + \frac{m}{2} -\veo\norm{x^-}^{}_1\\
&\,\le\, -\frac{\varrho\veo}{m} + \frac{m}{2} -\frac{\veo(1-\delta)}{2}\norm{x}^{}_1
\quad \forall\, (x,u)\in\bigl(\cK_0^+\setminus\cK_\delta^+\bigr)\times\varDelta\,,\\
\end{aligned}
\end{equation}
and
\begin{equation}\label{PL4.2C}
F_2(x,u) \,\le\,
-\frac{\varrho\veo}{m} - \veo\langle e,x \rangle + \veo \langle e,x \rangle
\,\le\, -\frac{\varrho\veo}{m} \qquad \forall (x,u)\in\cK_\delta^+\times\varDelta\,,
\end{equation}
for any $\delta\in(0,1)$.

The choice of $\theta$ implies that $V_1\ge V_2^2$ on $\cK_0^-$.
Thus \cref{ELstableA1} holds by \cref{PL4.2A,PL4.2B}, while
\cref{PL4.2C} is equivalent to \cref{ELstableA2}.
\end{proof}

\subsubsection{The case of networks with class-dependent service rates}

Consider the class of networks described in
\cref{S2.4.2}.
Such networks have a
limiting diffusion with the same drift structure
studied in \cite{AHP18}, and that paper shows that when $\varrho>0$,
then the diffusion (and the prelimit) is uniformly
exponentially ergodic in the presence or absence of abandonment.
However, the proof of uniform exponential ergodicity of the prelimit
for models with class-dependent service rates
does not seem to carry through with the Lyapunov function used in
\cite{AHP18}.
Thus for the sake of proving the result for the $n^{\mathrm{th}}$ system in \cref{S5},
we adopt here the Lyapunov function in \cref{ED4.1C}.

\begin{lemma}\label{L4.3} 
Consider a network satisfying $\mu_{ij} = \mu_{i}$ for all $i\in\cI$, and $\varrho>0$.
Let $\delta\in(0,1)$, and
$\theta\ge\theta_0\df 2\,\frac{\mu\mx}{\mu\mn}$.
Then, \cref{ELstableA1,ELstableA2} hold.
\end{lemma}

\begin{proof}
A simple calculation using \cref{Edrift-i} shows that
\begin{equation}\label{PL4.3A}
\begin{aligned}
\frac{1}{\theta}\,F_1(x,u)
&\,=\, \frac{\varrho}{m}\sum_{i\in\cI}\psi'_\veo(-x_i)
+ \sum_{i\in\cI}\psi'_\veo(-x_i)\bigl(x_i - u_i^c \langle e,x\rangle^+\bigr)\,,\\[5pt]
F_2(x,u) &\,=\,-\frac{\varrho}{m}\sum_{i\in\cI}\psi'_\veo(x_i)
- \sum_{i\in\cI}\psi'_\veo(x_i)\bigl(x_i - u_i^c \langle e,x\rangle^+\bigr)\,.
\end{aligned}
\end{equation}

Using \cref{PL4.3A}, we obtain
\begin{equation*}
\frac{1}{\theta}\,F_1(x,u) \,\le\, \varrho \veo  + \frac{m}{2}
- \frac{\veo}{2}\norm{x}^{}_1 \qquad \forall (x,u)\in \cK_0^-\times \varDelta\,,
\end{equation*}
Therefore, \cref{ELstableA1} holds
on $\cK_0^-\times \varDelta$ by this inequality and the choice of $\theta$.

On $\cK_0^+\times \varDelta$, the equations in \cref{PL4.3A} are identical
to the corresponding ones for a network with a dominant server pool,
for which the result has already been established in \cref{L4.2}.  
This completes the proof.
\end{proof}

\section{Uniform exponential ergodicity of the \texorpdfstring{$n^{\mathrm{th}}$}{}
system}\label{S5}

In this section we show that if $\varrho_n>0$ then the
prelimit of a network with a dominant server pool, or with
class-dependent service rates,
is uniformly exponentially ergodic and the invariant distributions
have exponential tails.

Recall that $\{\Tilde\mu_i\,,\;i\in\cI\}$ are the elements of the
diagonal matrix $B_1^n$ in \cref{Ebn}.
Throughout this section $V$ denotes the function in \cref{ED4.1C}, with
$\veo$ given by
\begin{equation}\label{Eveon}
\veo\,=\, \veo_n\,\df\, \frac{\varrho_n}{3m}
\Biggl(\sum_{i\in\cI}\frac{1}{n}
\frac{\lambda^n_i(3\Tilde\mu^n_i+2)}{(\Tilde\mu^n_{i})^2}\Biggr)^{-1}
\,\exp\Biggl(-\frac{1}{\sqrt n}\,\sum_{i\in\cI}\frac{1}{\Tilde\mu^n_{i}}\Biggr)\,.
\end{equation}

Recall the definition of the operator
$\widehat\cL_{\Hat{z}}^n$ in \cref{E-Gn}, and the definitions of
$\sS^n$ and $\tcZn(\Hat{x})$ in \cref{EsSn,D2.3}.
We start with the following simple assertion.

\begin{lemma}\label{L5.1}
Let $V$ be the function in \cref{ED4.1C} with $\veo$ as in \cref{Eveon},
and $\theta$ fixed at some value.
Suppose that for any $\delta\in(0,1)$ there exist positive constants
$c_0$ and $c_1$ such that the drift $b^n$ in \cref{Ebn} satisfies
\begin{equation}\label{EL5.1A}
\begin{aligned}
\bigl\langle b^n(\Hat{x},\Hat{z}),\nabla V(\Hat{x}) \bigr\rangle &\,\le\, 
c_0 - \veo c_1 \norm{x}^{}_1 V(\Hat{x})
 \qquad\forall\,\Hat{x}\in\sS^n\setminus\cK_\delta^+\,,
 \ \forall\,\Hat{z}\in\tcZn(\Hat{x})\,, \\[5pt]
\bigl\langle b^n(\Hat{x},\Hat{z}),\nabla V_2(\Hat{x}) \bigr\rangle &\,\le\,
- \frac{\varrho_n\veo_n}{2m} V_2(\Hat{x})
\qquad\mspace{44mu}
\forall\,\Hat{x}\in\sS^n\cap\cK_\delta^+\,,\ \forall\,\Hat{z}\in\tcZn(\Hat{x})\,.
\end{aligned}
\end{equation}
Then, there exists a constant $\widehat{C}_0$ such that
\begin{equation}\label{EL5.1B}
\widehat\cL_{\Hat{z}}^n V\bigl(\Hat{x}\bigr)
\,\le\, \widehat{C}_0 - \frac{\varrho_n\veo_n}{4m} V(\Hat{x})\qquad
\forall\,\Hat{x}\in\sS^n\,,\ \forall\,\Hat{z}\in\tcZn(\Hat{x})\,.
\end{equation}
\end{lemma}

\begin{proof}
A simple calculation shows that
\begin{equation*}
\int_0^1 (1-t)\,
\partial_{x_ix_i} V_2\Bigl(\Hat{x}\pm \tfrac{t}{\sqrt n} e_i\Bigr) \,\D{t}
\,\le\,
\frac{\veo_n^2}{2}\Biggl(\sum_{i\in\cI}
\frac{(3\Tilde\mu^n_i+1)}{(2\Tilde\mu^n_{i})^2}\Biggr)
\,\exp\Biggl(\frac{1}{\sqrt n}\,
\sum_{i\in\cI}\frac{1}{\Tilde\mu^n_{i}}\Biggr)\,V_2(\Hat{x})\,.
\end{equation*}
Thus, using \cref{Eident} to express the first and second order incremental quotients
in \cref{E-Gn}, we obtain
\begin{equation*}
\widehat\cL_{\Hat{z}}^n V_2\bigl(\Hat{x}\bigr)
\,\le\, \frac{\varrho_n\veo_n}{4m} V_2(\Hat{x})
+ \bigl\langle b^n(\Hat{x},\Hat{z}),\nabla V_2(\Hat{x}) \bigr\rangle\,.
\end{equation*}
The rest follows as in the proof of \cref{Lstable} by selecting $\delta$
sufficiently close to $1$.
\end{proof}

\begin{remark}
Recall \cref{Ebn}.
In direct analogy to \cref{R4.2}, if we let
$\zeta^n =\frac{\varrho_n}{m}  e+ (B_1^n)^{-1}\ell^n$, with $\varrho_n$
as in \cref{Evarrho}, then a mere translation of the origin of the form
$\tilde{X}^n = \Hat{X}^n + \zeta^n$
results in a diffusion of with the same drift
as \cref{Ebn}, except that
the vector $\ell^n$ gets replaced by $\ell^n=- \frac{\varrho_n}{m} (B_1)^n e$.
Therefore,  we may assume without any loss of
generality that
the drift in \cref{Ebn} takes the form
\begin{equation}\label{ER5.1A}
b^n(\Hat{x},\Hat{z}) \,=\, -\frac{\varrho_n}{m}B_1^n e
- B_1^n \bigl(\Hat{x}-\langle e,\Hat{x}\rangle^{+} u^c\bigr)
+ B_2^n u^s \langle e,\Hat{x}\rangle^{-}
+\Hat{\vartheta}^{n}(\Hat{x},\Hat{z}) \bigl(B_1^n u^c + B_2^n u^s\bigr)\,.
\end{equation}

Note that this centering
has the effect of translating the `equilibrium' allocations
$\overline{z}^n_{ij}$ given in \cref{Eequil}.
Since this translation is of $\order(n^{-\nicefrac{1}{2}})$, it has no effect
on the
results for large $n$. However, in the interest of providing precise estimates
we calculate the new values of $\overline{z}^n_{ij}$.
Note that $\langle e,\zeta^n\rangle=0$, and recall the map $\Phi$ in \cref{EPhi}.
Let $\Check{z}^n_{ij} = \Phi(\zeta^n,0)$.
Then, the centering of $\Hat{x}$ that results in \cref{ER5.1A}
is given by (compare with \cref{Eequil})
\begin{equation}\label{Eequil'}
\Bar{z}^n_{ij} \,=\, \frac{1}{n} \xi^*_{ij}N_j^n + \frac{\Check{z}^n_{ij}}{\sqrt n}\,,
\qquad \Bar{x}^n_i \,\df\, \sum_{j\in\cJ} \Bar{z}^n_{ij}\,.
\end{equation}
Throughout this section the family
$\{\overline{z}^n_{ij}\; (i,j)\in\cE\}$ is as given in \cref{Eequil'}.
\end{remark}
\begin{remark}
Recall \cref{D2.3,EhsX}. Let $\Hat{z}\in\tcZn(\Hat{x})$.
Then, $\Hat{\vartheta}^n(\Hat{x},\Hat{z})
= \Hat{\vartheta}^n_\ast(\Hat{x})=0$ for all $\Hat{x}\in\Breve{\sX}^{n}$,
and in view of \cref{ER5.1A},
for any $\Hat{x}\in\Breve{\sX}^{n}$, 
there exists $u=u(\Hat{x},\Hat{z})\in\varDelta$ such that
\begin{equation}\label{Ebnjwc}
b^n(\Hat{x},\Hat{z}) \,=\, -\frac{\varrho_n}{m}B_1^n e 
- B_1^n \bigl(\Hat{x}-\langle e,\Hat{x}\rangle^{+} u^c\bigr)
+ B_2^n u^s \langle e,\Hat{x}\rangle^{-}\,.
\end{equation}

In view of \cref{L5.1}, and using \cref{L4.1,L4.2,L4.3},
it is clear that Foster--Lyapunov equations for $\Lg_u$ carry over
to analogous equations for $\widehat\cL_{\Hat{z}}^n$ on $\Breve{\sX}^{n}$
uniformly over SWC policies.
However, even though $\Breve{\sX}^{n}$ fills the whole space as $n\to\infty$,
$b$ and $b^n$ differ in functional form when
$\Hat{\vartheta}^n(\Hat{x},\Hat{z})\ne0$, and this makes the stability
analysis of multiclass multi-pool networks much harder than the
`V'~network studied in \cite{AHP18}.
\end{remark}

\begin{notation}\label{D5.1}
Let $\veo_n$ and $\Bar{z}^n_{ij}$ as in \cref{Eveon,Eequil'}, respectively.
For a network with a dominant server pool as in \cref{S4.3}
define
\begin{equation}\label{ED5.1A}
n_0 \,\df\, \max\,\biggl\{n\in\NN\,\colon
\frac{1}{\sqrt n}\ge \veo_n\,\min_{i\in\cI}\, \Bar{z}^n_{i1} \biggr\}\,,
\end{equation}
while for network with class-dependent service rates, we let
\begin{equation}\label{ED5.1B}
n_0 \,\df\, \max\,\biggl\{n\in\NN\,\colon 
\frac{1}{\sqrt n} \ge \frac{\veo_n}{2m}\,\min_{i\sim j}\,\Bar{z}^n_{ij}\biggr\}\,.
\end{equation}
Since $\{\veo_n\}$ and $\{\Bar{z}^n_{i1}\}$ are bounded away from $0$
by the convergence of the parameters in \cref{EHW},
the number $n_0$ is finite.
\end{notation}

The next theorem is the main result for the uniform exponential ergodicity of the prelimit processes.
Recall \cref{Evarrho}. Notice the similarities between the results in \cref{T5.1} for the diffusion-scaled processes and the results in \cref{Tstable,R4.1} for the diffusion limit of the networks under consideration.
\begin{theorem}\label{T5.1}
 Assume that $\varrho_n>0$, and let $n_0$ be as in \cref{D5.1}.
Then the prelimit dynamics of any network with a dominant server pool
or with  class-dependent service rates are uniformly exponentially ergodic
and the invariant distributions have exponential tails for all $n> n_0$.
In particular, due to the convergence of the parameters,
there exists $\widehat{C}_0$ independent of $n$ such that 
\begin{equation}\label{ET5.1A}
\widehat\cL_{\Hat{z}}^n V\bigl(\Hat{x}\bigr)
\,\le\, \widehat{C}_0 - \frac{\varrho_n\veo_n}{4m} V(\Hat{x})\qquad
\forall\,\Hat{x}\in\sS^n\,,\ \forall\,\Hat{z}\in\tcZn(\Hat{x})\,.
\end{equation}
where $V$ and $\veo$ are as in \cref{ED4.1C,Eveon}. 

In addition, with $P_t^{n,\Hat{z}}$ and $\uppi^n_{\Hat{z}}$
denoting respectively the transition probability and the stationary distribution of
$\Hat{X}^{n}(t)$ under a policy $\Hat{z}\in\tfZn$, there exist positive constants
$\gamma$ and $C_\gamma$ not depending on $n\ge 0$ or $\Hat{z}$, such that
\begin{equation} \label{ET5.2B}
\bnorm{ P^{n,\Hat{z}}_t(\Hat{x},\cdot\,)-\uppi^n_{\Hat{z}}(\cdot)\,}_{V}\,\le\,
C_\gamma V(\Hat{x})\, \E^{-\gamma t}\,,
\qquad \forall\,\Hat{x}\in\sX^n\,,\ \forall\,t\ge0\,.
\end{equation}
\end{theorem}

\begin{proof}
In \cref{L5.4,L5.5} in the section which follows,  we establish
\cref{EL5.1A} for these networks.
Thus the proof of \cref{ET5.1A} follows directly from \cref{L5.1}.

Since the process $\Hat{X}^n$ is irreducible and aperiodic under
any stationary Markov scheduling $\Hat{z}\in\tfZn$ (see \cref{D2.3}),
a convergence property completely analogous to \cref{ER4.1A} follows from
\cref{ET5.1A}. The proof of this fact is identical to \cite[Theorem~2.1(b)]{AHP18}. 
\end{proof}

\begin{remark} \label{R5.3} 

Using \cref{ETstableB} and \cref{ET5.2B}, it is clear that under any scheduling policy $\Hat{z}\in\tfZn$ with a corresponding control $v$, the stationary distribution  of the diffusion-scaled process $\Hat{X}^{n}(t)$ converges to that of the limiting diffusion 
$\process{X}$ for the two classes of networks, that is, 
\begin{equation}\label{ER5.3A}
\uppi^n_{\Hat{z}}(\cdot) \,\to\, \uppi_v(\cdot)\,, \quad \text{as}\quad n \to \infty. 
\end{equation}
That is, the interchange of limits property holds. 

%
\end{remark} 

\subsection{Four technical lemmas}
In this section, we establish the technical results
used in the proof of \cref{T5.1}.

Let
\begin{equation}\label{EtsX}
\widetilde{\sX}^n\,\df\,
\bigl\{\Hat{x} \in \sS^n\,\colon  \Hat\vartheta_*^n(\Hat{x})\ne0  \bigr\}\,,
\end{equation}
with $\Hat\vartheta_*^n$ as in \cref{D2.3}.
As seen in \cref{S2.3}, the set $\Breve{\sX}^{n}$
in \cref{EhsX} is contained in $\sS^n\setminus\widetilde{\sX}^n$.
In establishing \cref{EL5.1A} on $\sS^n\setminus\widetilde{\sX}^n$, the results in
\cref{S4} pave the way, since the drift of of the
controlled generator $\widehat\cL_{\Hat{z}}^n$ over the class of SWC
stationary Markov policies $\tfZn$ (see \cref{Ebnjwc})
has the same functional form as the drift of the diffusion in \cref{ELg}.
So it remains to establish \cref{EL5.1A} in $\widetilde{\sX}^n$.
We start by establishing a bound for
$\Hat\vartheta^n$ in \cref{ED2.2B} over all SWC policies.

As done earlier in the interest of notational economy,
we suppress the dependence on $n$ in the diffusion
scaled variables $\Hat{x}^n$ and $\Hat{z}^n$ in \cref{ED2.1A}. 

\begin{lemma}\label{Lnice}
There exists a number $\varkappa^n_\circ<1$ depending only on
the parameters of the network such that
\begin{equation}\label{ELniceA}
\Hat{\vartheta}^{n}(\Hat{x},\Hat{z}) \,=\, \Hat\vartheta^n_*(\Hat{x})
\,\le\, \varkappa^n_\circ
\bigl(\norm{\Hat{x}^+}_1\wedge \norm{\Hat{x}^-}_1\bigr)
\quad \forall\,\Hat{z}\in\tcZn(\Hat{x})\,.
\end{equation}
In addition, due to the convergence of the parameters in \cref{EHW},
such a constant $\varkappa_\circ<1$ may be selected which does
not depend on $n$.
\end{lemma}

Before proceeding to the proof of \cref{Lnice}, we provide an interpretation of \cref{ELniceA}. Recall from \cref{Ebal} that $\hat{x}_i = \frac{1}{\sqrt{n}}(x_i - \sum_{j \in \cJ} \xi^*_{ij} N_j^n)$, and observe that $\hat{x}_i$ is positive if the total number of class $i$ customers exceeds the total number of servers assigned to class $i$ in the fluid equilibrium $\bigl(\sum_{j \in \cJ} \xi^*_{ij} N_j^n\bigr)$ and is negative otherwise. This means that the quantity $\norm{\Hat{x}^+}_1\wedge \norm{\Hat{x}^-}_1 $ on the right hand side of \cref{ELniceA} represents the minimum of the total number of customers in the queues and the total number of idle servers if the servers are assigned to customer classes according to the fluid equilibrium. Recall also that $\hat{\vartheta}^n_*$ represents the minimum of the total number of customers in the queues and the total number of idle servers under a SWC policy. The result in \cref{Lnice} is now clear, that is, the minimum of the total number of customers in the queues and idle servers is smaller under a SWC policy compared to a policy that assigns servers according to the fluid equilibrium.

\begin{proof}\label{PLnice}
Let $\Hat{x}\in\widetilde{\sX}^n$, $\Hat{z}\in\tcZn(\Hat{x})$, and define
\begin{equation*}
\widetilde\cJ\,\df\, \Biggl\{j\in\cJ\,\colon \sum_{j\in\cJ(i)} \Hat{z}_{ij}<0
\Biggr\}\,,\quad\text{and\ \ }
\widetilde\cI\,\df\, \bigl\{i\in\cI\,\colon (i,j)\in\cE\text{\ for some\ } j\in
\widetilde\cJ\bigr\}\,,
\end{equation*}
and $\widetilde\cE \df \bigl\{(i,j)\in\cE\,\colon
(i,j)\in \widetilde\cI\times\widetilde\cJ\bigr\}$.
Work conservation implies that $x_i^n = \sum_{j\in\cJ(i)} z_{ij}^n$
for all $i\in\widetilde\cI$.
Let $\Hat\imath\in\cI$ be such that $\Hat{q}_{\Hat\imath}>0$, and consider
the unique path (since the graph of the network is a tree)
connecting $\Hat\imath$ and $\widetilde\cI$, that is,
a path $\Hat\imath \rightarrow j_1 \rightarrow i_1 \rightarrow j_2 \dotsc \rightarrow j_k \rightarrow \Tilde\imath$,
with $j_\ell \in \cJ\setminus\widetilde\cJ$ for $\ell=1,\dotsc,k$,
$i_\ell\in \cI\setminus\widetilde\cI$ for $\ell=1,\dotsc,k-1$,
and $\Tilde\imath\in\widetilde\cI$.
We claim that $z_{\Tilde\imath,j_k}^n=0$, or equivalently,
that $\Hat{z}_{\Tilde\imath,j_k} = - \nicefrac{\Bar{z}_{ij}^n}{\sqrt n}$,
with $\Bar{z}_{ij}^n$ as defined in \cref{Eequil'}.
If not, then we can move a job of class $\Tilde\imath$ from pool
$j_k$ to some pool in $\widetilde\cJ$, and proceeding
along the path to place one additional job from class $\Hat\imath$ into service,
thus contradicting the hypothesis that $\Hat{z}\in\tcZn(\Hat{x})$.
Removing all such paths, we are left with
a strict subnetwork (possibly disconnected)
 $\cG_\circ=\bigl(\cI_\circ\cup\cJ_\circ, \cE_\circ)$,
with $\cI_\circ\supset\widetilde\cI$,
$\cJ_\circ\supset\widetilde\cJ$,
and $\cE_\circ \df \bigl\{(i,j)\in\cE\,\colon
(i,j)\in \cI_\circ\times\cJ_\circ\bigr\}$,
such that
\begin{equation}\label{PLniceB}
x_i^n = \sum_{j\in\cJ(i)\cap\cJ_\circ} z_{ij}^n\,,\quad\forall\,
i\in\cI_\circ\,.
\end{equation}
Let $\cE_\circ' \df
\bigl(\cI_\circ\times (\cJ\setminus\cJ_\circ)\bigr)\cap\cE$.
By \cref{PLniceB} we have
\begin{equation*}
\sum_{(i,j)\in\cE_\circ'} z^n_{ij} \,=\,0\,.
\end{equation*}
Thus we have
\begin{equation}\label{PLniceC}
\norm{\Hat{x}^-}_1 \,\ge\, -\sum_{i\in\cI_\circ} \Hat{x}_i
\,=\, \sqrt{n} \sum_{(i,j)\in\cE_\circ'}\Bar{z}_{ij}^n -
\sum_{(i,j)\in\cE_\circ} \Hat{z}_{ij}
\,\ge\, \sqrt{n}\sum_{(i,j)\in\cE_\circ'}\Bar{z}_{ij}^n
+ \Hat\vartheta^n_*(\Hat{x})
\end{equation}
by the construction above.
By \cref{PLniceC}, we obtain
\begin{equation}\label{PLniceD}
\begin{aligned}
\Hat\vartheta^n_*(\Hat{x}) &\,\le\,
\frac{\sum_{(i,j)\in\cE_\circ} \Hat{z}_{ij}}
{\sqrt{n}\sum_{(i,j)\in\cE_\circ'}\Bar{z}_{ij}^n -
\sum_{(i,j)\in\cE_\circ} \Hat{z}_{ij}}\, \sum_{i\in\cI_\circ} \Hat{x}_i\\
&\,\le\, -\frac{\sum_{(i,j)\in\cE_\circ} \Bar{z}^n_{ij}}
{\sum_{(i,j)\in\cE_\circ}\Bar{z}_{ij}^n +
\sum_{(i,j)\in\cE_\circ'} \Bar{z}^n_{ij}}\, \sum_{i\in\cI_\circ} \Hat{x}_i\,.
\end{aligned}
\end{equation}
Similarly, 
\begin{equation}\label{PLniceE}
\norm{\Hat{x}^+}_1 \,\ge\, \sum_{i\in\cI\setminus\cI_\circ} \Hat{x}_i
\,\ge\, \sqrt{n}\sum_{(i,j)\in\cE_\circ'}\Bar{z}_{ij}^n
+ \Hat\vartheta^n_*(\Hat{x})\,.
\end{equation}
Using the bound $\Hat\vartheta^n_*(\Hat{x})\le
\sqrt{n}\sum_{(i,j)\in\cE_\circ} \Bar{z}^n_{ij}$ we obtain from
\cref{PLniceE} that
\begin{equation}\label{PLniceF}
\begin{aligned}
\Hat\vartheta^n_*(\Hat{x})&\,\le\,
\frac{\Bigl(\sum_{i\in\cI\setminus\cI_\circ} \Hat{x}_i
-\sqrt{n}\sum_{(i,j)\in\cE_\circ'}\Bar{z}_{ij}^n
\Bigr)\wedge \sqrt{n}\sum_{(i,j)\in\cE_\circ} \Bar{z}^n_{ij}}
{\sum_{i\in\cI\setminus\cI_\circ} \Hat{x}_i}\,
\sum_{i\in\cI\setminus\cI_\circ} \Hat{x}_i\\[5pt]
&\,\le\,
\frac{\sum_{(i,j)\in\cE_\circ} \Bar{z}^n_{ij}}
{\sum_{(i,j)\in\cE_\circ} \Bar{z}^n_{ij}
+\sqrt{n}\sum_{(i,j)\in\cE_\circ'}\Bar{z}_{ij}^n}\,
\sum_{i\in\cI\setminus\cI_\circ} \Hat{x}_i\,.
\end{aligned}
\end{equation}

It should be now clear how to construct
$\varkappa_\circ^n$.
For any given subset $\cJ'\subsetneq \cJ$,
let
\begin{equation*}
\cI_{\cJ'}\,\df\, \cup_{j\in\cJ'}\, \cI(j)\,,\qquad
\cE_{\cJ'} \,\df\, \bigl\{(i,j)\in\cE\,\colon
(i,j)\in \cI_{\cJ'}\times\cJ'\bigr\}\,,
\end{equation*}
and
$\cE'_{\cJ'} \df \bigl(\cI_{\cJ'}\times (\cJ\setminus\cJ')\bigr)\cap\cE$,
and define
\begin{equation*}
\varkappa_\circ^n \,\df\,
\max_{\cJ'\subsetneq \cJ}\;
\frac{\sum_{(i,j)\in\cE_{\cJ'}} \Bar{z}^n_{ij}}
{\sum_{(i,j)\in\cE_{\cJ'}}\Bar{z}_{ij}^n +
\sum_{(i,j)\in\cE'_{\cJ'}} \Bar{z}^n_{ij}}\,.
\end{equation*}
Then the result clearly follows from \cref{PLniceD,PLniceF}
since
\begin{equation*}
\norm{\Hat{x}^-}_1 \,\ge\, -\sum_{i\in\cI_\circ} \Hat{x}_i\,,\quad\text{and}\quad
\norm{\Hat{x}^+}_1 \,\ge\, \sum_{i\in\cI\setminus\cI_\circ} \Hat{x}_i\,.
\end{equation*}
This completes the proof.
\end{proof}

\begin{remark}
We provide an example of a SWC policy for the `N' network with two classes of customers and two server pools where class 1 can be served by both server pools while class 2 can only served by pool 2 to explain the analysis in the proof of \cref{Lnice}. 
 The SWC policy is given by
\begin{align*}
z_{11}(x) & \,=\, x_1\wedge N^n_1 \\
z_{12} (x) &\,=\, 
\begin{cases}
(x_1 - N^n_1)^+ \wedge \xi^*_{12} N^n_{2} \qquad & \text{if } x_2 \ge \xi^*_{22} N^n_{2}\\
(x_1 - N^n_1)^+ \wedge (N_2^n - x_2) \qquad &\text{otherwise,}
\end{cases} \\
z_{22}(x) & \,=\, 
\begin{cases}
x_2 \wedge \xi^*_{22} N^n_{2} \qquad &\text{if } x_1 \ge N_1^n + \xi^*_{12} N^n_{2}\\
x_2\wedge\bigl(N_2^n - (x_1 - N_1^n)^+\bigr)\qquad & \text{otherwise.}
\end{cases}
\end{align*}
This is a priority policy in which customer class $1$ prefers pool $1$ over pool $2$.
 This means that customer class $1$ can use servers in pool $2$ only if there are no idle servers in pool $1$.

Recall from \cref{Edszx} and \cref{Ebal} that $\hat{q}_i \ge 0$ for all $i \in \cI$. Recall also that we are only considering work-conserving policies and that $\hat\vartheta^n$ represents the minimum of the total number of customers in the queues and the total number of idle servers. For the `N' network with two classes of customers and two server pools, we have the following cases:

\begin{itemize}
\item[Case 1]: $\hat{q}_1 > 0$ and $\hat{q}_2 \ge 0$, in which case  we have no idle servers and hence $\hat{\vartheta}^n_*(\hat{x}) = 0$.

\item[Case 2]: $\hat{q}_1 = 0$ and $\hat{q}_2 =0$. Again, this is a trivial case and means that we have no customers in the queues which implies that $\hat{\vartheta}^n_*(\hat{x}) = 0$.

\item[Case 3]: $\hat{q}_1 = 0$ and $\hat{q}_2 > 0$. This is actually the case that requires some analysis. Here again we have the following cases:
\begin{itemize}
\item[a]: $\hat{y}_1 = 0$ which means that there are no idle servers in pool 1 and hence $\hat{\vartheta}^n_*(\hat{x}) = 0$. This is because $\hat{q}_2 > 0$ which implies that $\hat{y}_2 = 0$ under work conservation.

\item[b]: $\hat{y}_1 > 0$ which means that there are some idle servers in pool $1$. This is the case analyzed in the proof of \cref{Lnice}. (Observe the following notation in the proof: $\widetilde{\mathcal{J}} = \{1\}$, $\widetilde{\mathcal{I}} = \{1\}$, $\Hat\imath=2$, the path is  $\text{class 2} \rightarrow \text{pool 2} \rightarrow \text{class 1}$, and $\cG_\circ = \{(1,1)\}$.) 
Note that since we are using a SWC policy, this means that no class $1$ customers are being served in pool $2$ because otherwise, one can move a class $1$ customer from pool $2$ to pool $1$ and get a smaller $\hat{\vartheta}^n$. Hence, 
\begin{equation*}
\hat{\vartheta}^n_*(\hat{x}) = \frac{1}{\sqrt{n}} \bigl((N_1^n - x_1) \wedge (x_2 - N_2^n)\bigr)\,,
\end{equation*}
where $(N_1^n - x_1)$ is the total number of idle server and $(x_2 - N_2^n)$ is the total number of customers in the queues. This is because the only pool having idle servers is pool $1$ and the only class having customers in the queue is class $2$ ($\hat{q}_1 = 0$).

Recall that $\hat{x}_i = \frac{1}{\sqrt{n}}(x_i - \sum_{j \in \cJ} \xi^*_{ij} N_j^n)$. 
This means in this case that
\begin{equation*}
\begin{aligned}
\hat{x}_1^+ = 0 \qquad\,&; \qquad \hat{x}_1^- = \frac{1}{\sqrt{n}}(N_1^n + \xi^*_{12} N_2^n - x_1)\\
\hat{x}_2^+ = \frac{1}{\sqrt{n}}(x_2 - \xi^*_{22} N_2^n) \qquad\,&;\qquad \hat{x}_2^- = 0\,.
\end{aligned}
\end{equation*}

Therefore, we have the following equation
\begin{equation*}
\norm{\Hat{x}^+}_1\wedge \norm{\Hat{x}^-}_1 = \frac{1}{\sqrt{n}}\Bigl((N_1^n + \xi^*_{12} N_2^n - x_1) \wedge (x_2 - \xi^*_{22} N_2^n)\Bigr).
\end{equation*}
Note also that \cref{ELP-id} ($\xi^*_{12} + \xi^*_{22} = 1$) implies that 
\begin{equation*}
N_1^n - x_1  \le x_2 - N_2^n \iff N_1^n + \xi^*_{12} N_2^n - x_1 \le  x_2 - \xi^*_{22} N_2^n\,.
\end{equation*}
It is now clear that there exists a constant $\varkappa_\circ^n < 1$ such that $\hat{\vartheta}^n_* < \varkappa_\circ^n \bigl(\norm{\Hat{x}^+}_1\wedge \norm{\Hat{x}^-}_1\bigr)$ where 
$$\varkappa_\circ^n = \frac{N_1^n - x_1}{N_1^n + \xi^*_{12} N_2^n - x_1} \vee \frac{x_2 - N_2^n}{x_2 - \xi^*_{22} N_2^n}.$$
\end{itemize}
This completes the analysis of all the cases.
\end{itemize}

\end{remark}
\begin{remark}\label{R5.2}
It is easy to see that the estimates of the bounds on $\Hat{\vartheta}^n$
can be improved. It is clear from \cref{PLniceC,PLniceD}, that
\begin{equation*}
\Hat\vartheta^n_*(\Hat{x}) \,\le\,
\Biggl(-\varkappa_0^n \sum_{i\in\cI_0}\Hat{x}_i\Biggr)
\wedge \Biggl(\sum_{i\in\cI\setminus\cI_\circ} \Hat{x}_i\Biggr)
\qquad \forall\,\Hat{x}\in\widetilde{\sX}^n\,.
\end{equation*}
Also, since there can be at most $\sum_{j\in\cJ} N_j^n$ idle servers,
it follows that
$\widetilde\varkappa_\circ^n\in(0,1)$, such that
\begin{equation*}
-\sum_{i\in\cI_0}\Hat{x}_i \,\ge\,
\widetilde\varkappa_\circ^n \norm{\Hat{x}^-}^{}_1
\qquad \forall\,\Hat{x}\in\widetilde{\sX}^n\,,
\end{equation*}
where the constant $\widetilde\varkappa_\circ^n\in(0,1)$ can be selected as
\begin{equation*}
\widetilde\varkappa_\circ^n \,\df\,
\Biggl(\sum_{j\in\cJ} N_j^n\Biggr)^{-1}
\min_{(i, j) \,\in\, \cE}\, \xi^*_{ij} N_{j}^n\,.
\end{equation*}
Due to the convergence of the parameters in \cref{EHW},
$\widetilde\varkappa_\circ^n$ is bounded away from $0$ uniformly in $n\in\NN$.
Note that for the `N' network this translates to
\begin{equation*}
\widetilde\varkappa_\circ^n = \frac{N_1^n \wedge \xi^*_{12} N_2^n \wedge \xi^*_{22} N_2^n}{N_1^n + N_2^n}\,. 
\end{equation*}

\end{remark}

Even though the `N'~network is a special case
of networks with a dominant server pool
we first establish the result for this network in \cref{L5.3}
in order to exhibit with simpler calculations how \cref{Lnice} is applied. 

Throughout the proofs of \cref{L5.3,L5.4,L5.5} we use the
functions (compare with \cref{EF})
\begin{equation*}
F_i^n(\Hat{x},\Hat{z}) \,\df\, \frac{1}{V_i(x)}\,
\bigl\langle b^n(\Hat{x},\Hat{z}), \nabla V_1(\Hat{x}) \bigr\rangle\,,\qquad
i=1,2\,,
\end{equation*}
and let $n_0$ be as in \cref{D5.1}.
Moreover, we suppress the dependence on $n$ in the variables
$\Hat{q}^n$, $\Hat{y}^n$, and $\Hat{\vartheta}^n$ in \cref{ED2.2A,ED2.2B},
and from $\veo_n$ in \cref{Eveon}. 

\subsubsection{The diffusion-scale of the `N' network}

We recall here \cite{Stolyar-15b}. Recall also that we label the non-leaf server node as $j=1$ without loss of generality and hence we present Stolyar's work in \cite{Stolyar-15b} using our notation.
In this work, Stolyar considers the `N'~network with
$\order(\sqrt n)$ safety staffing in pool $1$,
under the priority discipline that class $2$ has priority in pool $1$
and class $1$ prefers pool $2$, and shows tightness of the invariant distributions.
First note that for any stationary Markov scheduling policy $z$,
such that class $2$ has priority in pool $1$ we have
$z_{21}^n(x) = x_2^n\wedge N_1^n$,
and it is clear that such a policy is SWC.
The same applies to Markov policies under which
class $1$ prefers pool $2$ (here $z_{12}^n(x) = x_1^n\wedge N_2^n$).
As a result, SWC policies are more general than the particular policy considered in
\cite{Stolyar-15b}.
Recall that the matrices $B_1^n$ and $B_2^n$ in the drift \cref{ER5.1A}
are given by 
\begin{equation}\label{EBi}
B_1^n \,=\, \begin{pmatrix} \mu^n_{11} & 0\\[5pt] 0 & \mu_{21}^n\end{pmatrix}\,,
\qquad\text{and\ \ }
B_2^n  \,=\, \begin{pmatrix}0& \mu_{12}^n - \mu_{11}^n\\[5pt] 0 & 0\end{pmatrix}
\end{equation}
It is also worth noting here, that the spare capacity $\varrho_n$ of the
$n^{\mathrm{th}}$ system is given by
\begin{equation*}
\varrho_n \,=\, - \frac{1}{\sqrt{n}} \biggl(\frac{\lambda^{n}_1}{\mu^n_{11}}
+\frac{\lambda^{n}_2}{\mu^n_{21}}
-  \frac{\mu_{12}^{n} N_2^n+\mu_{11}^{n} \xi^*_{11}N_1^n}{\mu^n_{11}}
- \frac{\mu_{21}^{n} \xi^*_{21}N_1^n}{\mu^n_{21}}\biggr)\,,
\end{equation*}
with
\begin{equation*}
\xi^*_{11} \,=\, \frac{\lambda_1 - \mu_{12}\nu_2}{\mu_{11} \nu_1}\,,
\quad\text{and}\quad
\xi^*_{21} \,=\, \frac{\lambda_2}{\mu_{21} \nu_1}\,.
\end{equation*}
This is clear by \cref{Eelln,EBi,Evarrho},
together with \cite[Equation (2)]{AP18}. 
We let $\eta^n \df \frac{\mu_{12}^n}{\mu_{11}^n}$.

\begin{lemma}\label{L5.3} 
Consider the `N'~network, and assume that $\varrho_n>0$.
Then for any
$\theta\ge\theta_0^n\df
\frac{\mu_{11}^n\, \vee\, \mu_{21}^n}{\mu_{11}^n\, \wedge \,\mu_{21}^n}$,
and $\delta\in(0,1)$,  there exist positive constants
$c_0$ and $c_1$ such that \cref{EL5.1A} holds  for all $n\ge n_0$.
\end{lemma}

\begin{proof}
Recall here that $j=1$ is the non-leaf server pool. As discussed earlier, it suffices
to establish \cref{EL5.1A} in $\widetilde{\sX}^n$.
It is clear that $\Hat{x}_1^-=\Hat{y}_2 + \sqrt{n}\Bar{z}_{11}$, and
$\Hat{x}_2=\Hat{q}_2 + \sqrt{n}\Bar{z}_{11}$
for all $\Hat{z}\in\hcZn(\Hat{x})$ and $\Hat{x}\in\widetilde{\sX}^n$,
with $\widetilde{\sX}^n$ as defined in \cref{EtsX}.
Hence $u_1^c = 0$, $u_2^c = 1$, $u_2^s = 1$, and  $u_1^s = 0$. Note here that SWC policies are interpreted through $u^c$ and $u^s$ as follows: idle servers are only allowed in pool 1 and customer queues are only allowed in class 2 which means that class $1$ must use all the servers in pool $2$ before using servers in pool $1$.
Also, by the definitions of $\psi_\veo$ and $n_0$ we have
\begin{equation}\label{PL5.3A}
\psi'_\veo(\Hat{x}_1)\,=\,\psi'_\veo(-\Hat{x}_2) \,=\,0
\qquad\forall\,\Hat{x}\in\widetilde{\sX}^n\,,\ \forall\,n\ge n_0\,.
\end{equation}

By \cref{ER5.1A,EBi}, we have
\begin{equation}\label{PL5.3B}
\begin{aligned}
\frac{1}{\theta}\,F^n_1(\Hat{x},\Hat{z})
&\,=\, \frac{\varrho_n}{2}\sum_{i\in\cI}\psi'_\veo(-\Hat{x}_i)
+ \sum_{i\in\cI}\psi'_\veo(-\Hat{x}_i)\bigl(\Hat{x}_i
- u_i^c \langle e,\Hat{x}\rangle^+\bigr) - \psi'_\veo(-\Hat{x}_1)(\eta^n - 1)
\langle e,\Hat{x}\rangle^-\\
&\mspace{320mu} - \Hat{\vartheta}^n \Bigl(\psi'_\veo(-\Hat{x}_2) 
 + \psi'_\veo(-\Hat{x}_1)(\eta^n - 1)\Bigr)\,, 
\end{aligned}
\end{equation}
\begin{equation}\label{PL5.3C}
\begin{aligned}
F_2^n(\Hat{x},\Hat{z})
&\,=\, -\frac{\varrho_n}{2}\sum_{i\in\cI}\psi'_\veo(\Hat{x}_i)
- \sum_{i\in\cI}\psi'_\veo(\Hat{x}_i)\bigl(\Hat{x}_i
- u_i^c \langle e,\Hat{x}\rangle^+\bigr) + \psi'_\veo(\Hat{x}_1)(\eta^n - 1)
\langle e,\Hat{x}\rangle^-\\
&\mspace{360mu} + \Hat{\vartheta}^n \Bigl(\psi'_\veo(\Hat{x}_2) 
+ \psi'_\veo(\Hat{x}_1)(\eta^n - 1)\Bigr)\,.
\end{aligned}
\end{equation}

Using the fact that
$\Hat{x}_1^- - \Hat{x}_2 = \Hat{y}_2 - \Hat{q}_2$,
and $\Hat{\vartheta}^n=\Hat{q}_2$ when $\langle e,\Hat{x}\rangle\le 0$, we obtain
from \cref{PL5.3A,PL5.3B} that
\begin{equation}\label{PL5.3D}
\begin{aligned}
\frac{1}{\theta}\,F_1^n(\Hat{x},\Hat{z}) &\,=\,
\frac{\varrho_n}{2}\psi'_\veo(-\Hat{x}_1)
- \psi'_\veo(-\Hat{x}_1) \Hat{x}_1^-
- \psi'_\veo(-\Hat{x}_1)(\eta^n - 1) \langle e,\Hat{x}\rangle^-
- \psi'_\veo(-\Hat{x}_1)(\eta^n - 1)\Hat\vartheta^n\\
&\,=\,\veo\Bigl(\frac{\varrho_n}{2} - \Hat{x}_1^-
- (\eta^n - 1) \bigl(\Hat{x}_1^- - \Hat{x}_2\bigr)
- (\eta^n - 1) \Hat\vartheta^n\Bigr)\\
&\,=\,\veo\Bigl(\frac{\varrho_n}{2} - \Hat{x}_1^- - (\eta^n-1) \Hat{y}_2\Bigr)\\
&\,\le\,\veo\Bigl(\frac{\varrho_n}{2} - \Hat{x}_1^- + (1-\eta^n)^+ \Hat{x}_1\Bigr)\\
&\,\le\,\veo\Bigl(\frac{\varrho_n}{2} - (\eta^n \wedge 1)\Hat{x}_1^- \Bigr)\,\\
&\,\le\, \veo \frac{\varrho_n}{2} - \frac{\veo}{2}(\eta^n \wedge 1)\norm{\Hat{x}}^{}_1
\qquad \forall\,(\Hat{x},\Hat{z})\in
\bigl(\widetilde{\sX}^n\cap\cK_0^-\bigr)\times \hcZn(\Hat{x})\,,
\ \forall\, n\ge n_0\,.
\end{aligned}
\end{equation}
Similarly, from \cref{PL5.3C}, we obtain
\begin{equation}\label{PL5.3E}
\begin{aligned}
F_2^n(\Hat{x},\Hat{z}) &\,=\,-\frac{\varrho_n}{2}
\sum_{i\in\cI}\psi'_\veo(\Hat{x}_i)
- \sum_{i\in\cI}\psi'_\veo(\Hat{x}_i)\Hat{x}_i
+ \psi'_\veo(\Hat{x}_1)(\eta^n - 1) \langle e,\Hat{x} \rangle ^-\\
&\mspace{250mu}+ \Hat{\vartheta}^n \Bigl(\psi'_\veo(\Hat{x}_2)
+ \psi'_\veo(\Hat{x}_1)(\eta^n - 1 )\Bigr)\\
&\,\le\, 1 - \veo \norm{\Hat{x}^+}^{}_1
+ \veo \varkappa_0 \Bigl(\norm{\Hat{x}^-}^{}_1\wedge\norm{\Hat{x}^+}^{}_1 \Bigr)\\
&\,\le\, 1 - \veo (1-\varkappa_0 )\norm{\Hat{x}^+}^{}_1
\qquad \forall\,(\Hat{x},\Hat{z})\in
\bigl(\widetilde{\sX}^n\cap\cK_0^-\bigr)\times \hcZn(\Hat{x})\,,
\ \forall\, n\ge n_0\,,
\end{aligned}
\end{equation}
where we also use \cref{ED4.1D,Lnice}.

We continue with the estimate on $\cK_0^+$. We have 
\begin{equation}\label{PL5.3F}
\begin{aligned}
\frac{1}{\theta}\,F_1^n(\Hat{x},\Hat{z})
&\,\le\,\frac{\varrho_n}{2}\psi'_\veo(-\Hat{x}_1)
- \psi'_\veo(-\Hat{x}_1) \Hat{x}_1^-
- \psi'_\veo(-\Hat{x}_1)(\eta^n - 1)\Hat\vartheta^n\\
&\,=\,\veo\Bigl(\frac{\varrho_n}{2} - \Hat{x}_1^-
- (\eta^n - 1) \Hat\vartheta^n\Bigr)\\
&\,\le\,\veo\Bigl(\frac{\varrho_n}{2} - (\eta^n \wedge 1)\Hat{x}_1^- \Bigr)
\qquad \forall\,(\Hat{x},\Hat{z})\in
\bigl(\widetilde{\sX}^n\cap\cK_0^+\bigr)\times \hcZn(\Hat{x})\,,
\ \forall\, n\ge n_0\,,
\end{aligned}
\end{equation}
where in the last inequality we also use \cref{Lnice}.
 
We break the estimate of $F_2^n$ in two parts.
First, using \eqref{ED4.1D}, \eqref{ED4.2A}, \eqref{PL5.3A} and \cref{Lnice}, we obtain
\begin{equation}\label{PL5.3G}
\begin{aligned}
F_2^n(\Hat{x},\Hat{z}) &\,\le\, -\frac{\varrho_n\veo}{2} -\veo\Hat{x}_2
+ \veo \langle e,\Hat{x} \rangle
+ \veo \varkappa_0  \Bigl(\Hat{x}_1^-\wedge\Hat{x}_2 \Bigr)\\
&\,\le\, -\frac{\varrho_n\veo}{2} -\veo(1-\varkappa_0)\Hat{x}_1^-\\
&\,\le\,
\begin{cases}
-\frac{\varrho_n\veo}{2} -\frac{\veo(1-\delta)}{2}\,
(1-\varkappa_0)\norm{\Hat{x}}^{}_1
&\text{for\ } (\Hat{x},\Hat{z})\in
\bigl(\widetilde{\sX}^n\cap(\cK_0^+\setminus\cK_\delta^+)\bigr)
\times \hcZn(\Hat{x})\\[5pt]
-\frac{\varrho_n\veo}{2} &
\text{for\ } (\Hat{x},\Hat{z})\in
\bigl(\widetilde{\sX}^n\cap\cK_\delta^+\bigr)\times \hcZn(\Hat{x}). 
\end{cases}
\end{aligned}
\end{equation}
Thus, \cref{EL5.1A} follows
by \cref{PL5.3D,PL5.3E,PL5.3F,PL5.3G,ED4.2A}.
This completes the proof.
\end{proof}

\subsubsection{The diffusion scale of networks with a dominant pool}

We describe these networks exactly as in \cref{S4.3} where the dominant
 server pool is $j = 1$.
We first note that the spare capacity $\varrho_n$ of the
$n^{\mathrm{th}}$ system is given by
\begin{equation*}
\varrho_n \,=\, - \frac{1}{\sqrt{n}}\Biggl(\sum_{i\in\cI}\frac{\lambda_i^n}{\mu_{i1}^n}
-  \sum_{i\in\cI} \sum_{j\in\cJ(i)}
\frac{\mu_{ij}^n}{\mu_{i1}^n}\xi^*_{ij}N_j^n \Biggr)\,,
\end{equation*}
where $\xi^*_{ij}$ satisfies
\begin{equation*}
\sum_{j\in\cJ(i)}\mu_{ij} \xi^*_{ij}\nu_j = \lambda_i \,.
\end{equation*}
This is again due to \cref{Eelln,ELP-id,Evarrho,Edrift-i}.

Recall from \cref{ER5.1A} that the drift reduces to the following form:
\begin{equation} \label{Ebn-M}
\begin{aligned}
b_i^n(\Hat{x},\Hat{z}) &\,=\, -\frac{\varrho_n}{m}\mu_{i1}^n
- \mu_{i1}^n\bigl(\Hat{x}_i - u_i^c \langle e,\Hat{x}\rangle^+\bigr)
+ \sum_{j\in\cJ_1(i)}\mu_{i1}^n\bigl(\eta_{ij}^n - 1\bigr)u_{j}^s
\langle e,\Hat{x} \rangle^- \\
&\mspace{220mu} +\Hat{\vartheta}^n
\Biggl( \mu_{i1}^n u_i^c + \sum_{j\in\cJ_1(i)}\mu_{i1}^n
\bigl(\eta_{ij}^n - 1\bigr)u_{j}^s\Biggr)\,,
\qquad i\in\cI\,,
\end{aligned}
\end{equation}
with $\eta_{ij}^n \df \frac{\mu_{ij}^n}{\mu_{i1}^n}$
for $j\in\cJ_1(i) \df \cJ(i) \setminus \{1\}$ and $i\in\cI$.
In analogy to \cref{S4.3}, we define
We define
\begin{equation*}
\Bar\eta_n \,\df\,  \max_{i\in\cI}\,\max_{j\in\cJ_1(i)}\, \eta^n_{ij}\,,
\quad\text{and\ \ }
\underline\eta_n \,\df\, \min_{i\in\cI}\,\min_{j\in\cJ_1(i)}\, \eta^n_{ij}\,.
\end{equation*}

\begin{lemma}\label{L5.4}
Consider a network with a dominant server pool, and assume $\varrho_n>0$.
Then for any
$\theta\ge\theta_0^n\df 2\,\frac{\max_i{\mu_{i1}^n}}{\min_i{\mu_{i1}^n}}$,
and $\delta\in(0,1)$,  there exist positive constants
$c_0$ and $c_1$ such that \cref{EL5.1A} holds for all $n\ge n_0$.
\end{lemma}

\begin{proof}
Suppose $\Hat{x}\in\widetilde{\sX}^n$. 
A simple calculation using \cref{Ebn-M} shows that
 \begin{equation}\label{PL5.4A}
\begin{aligned}
\frac{1}{\theta}\,F_1^n(\Hat{x},\Hat{z})
\,=\, \frac{\varrho_n}{m}\sum_{i\in\cI}\psi'_\veo(-\Hat{x}_i)
&+ \sum_{i\in\cI}\psi'_\veo(-\Hat{x}_i)\bigl(\Hat{x}_i - u_i^c
\langle e,\Hat{x}\rangle^+\bigr) \\
&\mspace{10mu}-\sum_{i\in\cI}\sum_{j\in\cJ_1(i)} \psi'_\veo(-\Hat{x}_i)
\bigl(\eta_{ij}^n - 1\bigr)u_{j}^s \langle e,\Hat{x} \rangle^- \\
&\mspace{30mu}- \Hat{\vartheta}^n \Biggl( \sum_{i\in\cI}\psi'_\veo(-\Hat{x}_i)u_i^c
+ \sum_{i\in\cI}\sum_{j\in\cJ_1(i)}\psi'_\veo(-\Hat{x}_i)\bigl(\eta_{ij}^n
- 1\bigr)u_{j}^s \Biggr)\,,
\end{aligned}
\end{equation}
and
\begin{equation}\label{PL5.4B}
\begin{aligned}
F_2^n(\Hat{x},\Hat{z})
\,=\, -\frac{\varrho_n}{m}\sum_{i\in\cI}\psi'_\veo(\Hat{x}_i)
&- \sum_{i\in\cI}\psi'_\veo(\Hat{x}_i)\bigl(\Hat{x}_i - u_i^c
\langle e,\Hat{x}\rangle^+\bigr) \\
&\mspace{10mu}+\sum_{i\in\cI}\sum_{j\in\cJ_1(i)} \psi'_\veo(\Hat{x}_i)
\bigl(\eta_{ij}^n - 1\bigr)u_{j}^s \langle e,\Hat{x} \rangle^- \\
&\mspace{30mu}+ \Hat{\vartheta}^n \Biggl( \sum_{i\in\cI}\psi'_\veo(\Hat{x}_i)u_i^c
+ \sum_{i\in\cI}\sum_{j\in\cJ_1(i)}\psi'_\veo(\Hat{x}_i)
\bigl(\eta_{ij}^n - 1\bigr)u_{j}^s \Biggr)\,.
\end{aligned}
\end{equation}

By \cref{PL5.4A} we obtain
\begin{equation}\label{PL5.4C}
\begin{aligned}
\frac{1}{\theta}\,F_1^n(\Hat{x},\Hat{z}) &\,\le\,
\varrho_n\veo +\frac{m}{2} - \veo \sum_{i\in\cI} \Hat{x}_i^-
+ \veo\bigl(1 - \underline\eta_n\bigr)^+ \langle e,\Hat{x}\rangle^-
+\veo\bigl(1-\underline\eta_n\bigr)^+\Hat\vartheta^n\\
&\,\le\, \varrho_n\veo  +\frac{m}{2}
- \veo \sum_{i\in\cI} \Hat{x}_i^-  + \veo \bigl(1-\underline\eta_n\bigr)^+
\Biggl(\sum_{i\in\cI} \bigl(\Hat{x}_i^- - \Hat{x}_i^+\bigr)
+ \norm{\Hat{x}^-}^{}_1 \wedge \norm{\Hat{x}^+}^{}_1\Biggr)\\
&\,\le\, \varrho_n \veo  +\frac{m}{2}
- \veo \bigl(\underline\eta_n \wedge 1\bigr) \norm{\Hat{x}^-}^{}_1\\
&\,\le\, \varrho_n \veo  +\frac{m}{2}
- \frac{\veo(1-\delta)}{2}\,\bigl(\underline\eta_n \wedge 1\bigr)\norm{\Hat{x}}^{}_1
\quad \forall\,(\Hat{x},\Hat{z})\in
\bigl(\widetilde{\sX}^n\setminus \cK_\delta^+\bigr)\times \hcZn(\Hat{x})\,,
\ \forall\, n\in\NN\,,
\end{aligned}
\end{equation}
where we used \cref{ED4.1D} in the first inequality,
\cref{Lnice} in the second, and \cref{ED4.2A} in the fourth.

Next, we estimate a bound for $F_2^n(\Hat{x},\Hat{z})$.
Recall the definitions of $\cI_\circ$, $\cJ_\circ$, $\cE_\circ$,
and $\cE_\circ'$ in the proof of \cref{Lnice}.
Since $x\in\widetilde\sX^n$, we have $u_i^c=0$ for all $i\in\cI_\circ$,
and $u_j^s=0$ for all $j\in\cJ_\circ^{\mathsf c}$. SWC policies are interpreted here as follows: customer classes must use all the servers in the leaf pools available to them before using servers in pool $j=1$.
Additionally, $\Hat{x}_i \le -\sum_{(i,j)\in\cE_\circ'} \Bar{z}_{ij}$
for $i\in\cI_\circ$,
which implies that $\psi_\veo'(x_i) = 0$ for all $i\in\cI_\circ$
and $n>n_0$, by \cref{D5.1}. 
Hence,  since $\sum_{i\in\cI_\circ^{\mathsf c}} \Hat{x}_i>0$,
where $\cI_\circ^{\mathsf c}\equiv \cI\setminus\cI_\circ$, we have
\begin{equation}\label{PL5.4G}
\sum_{i\in\cI}\psi'_\veo(\Hat{x}_i)
\bigl(\Hat{x}_i - u_i^c \langle e,\Hat{x}\rangle^+\bigr)
\,\ge\,  \sum_{i\in\cI_\circ^{\mathsf c}}\psi'_\veo(\Hat{x}_i)\Hat{x}_i
- \veo \sum_{i\in\cI_\circ^{\mathsf c}} \Hat{x}_i - \veo \sum_{i\in\cI_\circ} \Hat{x}_i
\,\ge\,- \veo \sum_{i\in\cI_\circ} \Hat{x}_i
\end{equation}
by \cref{ED4.2D}.
Using \cref{PL5.4B} together with \cref{R5.2,PL5.4G}, we obtain
\begin{equation}\label{PL5.4H}
\begin{aligned}
F_2^n(\Hat{x},\Hat{z}) &\,\le\, -\frac{\varrho_n\veo}{m} 
+\veo \Hat\vartheta^n + \veo\sum_{i\in\cI_\circ}\Hat{x}_i
+ \Hat{\vartheta}^n \Biggl(\sum_{i\in\cI_\circ}\sum_{j\in\cJ_1(i)}\psi'_\veo(\Hat{x}_i)
\bigl(\eta_{ij}^n - 1\bigr)u_{j}^s \Biggr)\\
&\,\le\,
- \frac{\varrho_n\veo}{m}
+ \veo (1-\varkappa_\circ^n)\sum_{i\in\cI_\circ}\Hat{x}_i\\
&\,\le\,
\begin{cases}
- \frac{\varrho_n\veo}{m}-\frac{\veo(1-\delta)}{2}\,
(1-\varkappa_\circ^n)\,\widetilde\varkappa_\circ^n\,\norm{\Hat{x}}^{}_1
& \text{for\ } (\Hat{x},\Hat{z})\in
\bigl(\widetilde{\sX}^n\cap(\cK_0^+\setminus\cK_\delta^+)\bigr)\times \hcZn(\Hat{x})
\\[5pt]
-\frac{\varrho_n\veo}{2} & \text{for\ } (\Hat{x},\Hat{z})\in
\bigl(\widetilde{\sX}^n\cap\cK_\delta^+\bigr)\times \hcZn(\Hat{x})\,,
\end{cases}
\end{aligned}
\end{equation}
for all $n\ge n_0$.
Thus, the result follows by \cref{PL5.4C,PL5.4H}, noting also
that the choice of $\theta$ implies that $V_1\ge V_2^2$ on $\cK_0^-$.
\end{proof}

\subsubsection{The diffusion-scale of networks with class-dependent service rates}

Recall from \cref{S4.3} that the drift in \cref{ER5.1A} reduces to 
\begin{equation}\label{Ebn-i}
b^n(\Hat{x},\Hat{z}) \,=\, -\frac{\varrho_n}{m}B_1^n e
- B_1^n \bigl(\Hat{x}-\langle e,\Hat{x}\rangle^{+} u^c\bigr)
+\Hat{\vartheta}^{n}(\Hat{x},\Hat{z}) B_1^n u^c \,.
\end{equation}
where $B_1^n=\diag(\mu_{1}^n,\dotsc,\mu_{m}^n)$.
Thus, the spare capacity $\varrho_n$ is given by
\begin{equation*}
\varrho_n \,=\, - \frac{1}{\sqrt{n}}
\biggl(\sum_{i\in\cI}\frac{\lambda_i^n}{\mu_i^n}
-  \sum_{i\in\cI} \sum_{j\in\cJ(i)}\xi^*_{ij}N_j^n \biggr)\,.
\end{equation*}

\begin{lemma}\label{L5.5} 
Suppose that $\mu_{ij}^n = \mu_{i}^n$, 
for all $i\in\cI$, and $\varrho_n>0$.
Then, for any
$\theta\ge\theta_0^n\df 2\,\frac{\mu^n\mx}{\mu^n\mn}$,
and $\delta\in(0,1)$, the conclusions of \cref{L5.4} follow.
\end{lemma}

\begin{proof}
Suppose $\Hat{x}\in\widetilde{\sX}^n$.
A simple calculation using \cref{Ebn-i} shows that
\begin{align}
\frac{1}{\theta}\,F_1^n(\Hat{x},\Hat{z})
&\,=\, \frac{\varrho_n}{m}\sum_{i\in\cI}\psi'_\veo(-\Hat{x}_i)
+ \sum_{i\in\cI}\psi'_\veo(-\Hat{x}_i)
\bigl(\Hat{x}_i - u_i^c \langle e,\Hat{x}\rangle^+\bigr)
- \Hat{\vartheta}^n \sum_{i\in\cI}\psi'_\veo(-\Hat{x}_i)u_i^c\,,\label{PL5.5A}
\\
F_2^n(\Hat{x},\Hat{z})
&\,=\,-\frac{\varrho_n}{m}\sum_{i\in\cI}\psi'_\veo(\Hat{x}_i)
- \sum_{i\in\cI}\psi'_\veo(\Hat{x}_i)
\bigl(\Hat{x}_i - u_i^c \langle e,\Hat{x}\rangle^+\bigr)
+ \Hat{\vartheta}^n \sum_{i\in\cI}\psi'_\veo(\Hat{x}_i)u_i^c\,.\label{PL5.5B}
\end{align}

By \cref{PL5.5A}, we obtain
\begin{equation*}
\begin{aligned}
\frac{1}{\theta}\,F_1^n(\Hat{x},\Hat{z}) &\,\le\,
\varrho_n\veo  + \frac{m}{2} - \veo\norm{\Hat{x}^-}^{}_1 \\
&\,\le\, \varrho_n \veo  + \frac{m}{2} - \frac{\veo(1-\delta)}{2}\norm{\Hat{x}}^{}_1
\qquad \forall\,(\Hat{x},\Hat{z})\in
\bigl(\widetilde{\sX}^n\setminus\cK_\delta^-\bigr)\times \hcZn(\Hat{x})\,,
\ \forall\, n\ge n_0\,.
\end{aligned}
\end{equation*}

In computing the analogous bound to \cref{PL5.4H},
there is a difference here.
It is not the case here that $\psi_\veo'(x_i) = 0$ for all $i\in\cI_\circ$
and $n>n_0$.

So instead, recalling that $u_i^c=0$ for all $i\in \cI_\circ$,
and since $\Hat{x}\in\cK_0^+$, we write
\begin{equation}\label{PL5.5D}
\begin{aligned}
-\sum_{i\in\cI}\psi'_\veo(\Hat{x}_i)
\bigl(\Hat{x}_i - u_i^c \langle e,\Hat{x}\rangle^+\bigr)
+ \Hat{\vartheta}^n \sum_{i\in\cI}\psi'_\veo(\Hat{x}_i)u_i^c
&\,\le\,
-\sum_{i\in\cI}\psi'_\veo(\Hat{x}_i)\Hat{x}_i + \veo\langle e,\Hat{x}\rangle
+ \veo\Hat{\vartheta}^n
\\
&\,\le\,
\veo \Biggl(\Hat{\vartheta}^n - \sum_{\ell\in\cI_\circ} \Hat{x}_\ell^-\Biggr)
- \sum_{i\in\cI_\circ}\psi'_\veo(\Hat{x}_i)\Hat{x}_i\\
&\qquad
- \Biggl(\sum_{i\in\cI_\circ^{\mathsf c}}\psi'_\veo(\Hat{x}_i)\Hat{x}_i
-\veo \sum_{i\in\cI_\circ^{\mathsf c}} \Hat{x}_i\Biggr)\,.\\
\end{aligned}
\end{equation}
The third term on the right-hand side is nonpositive by \cref{ED4.2D}.
We also have 
\begin{equation}\label{PL5.5E}
\Hat{\vartheta}^n - \sum_{\ell\in\cI_\circ} \Hat{x}_\ell^-
\,=\, - \sqrt{n}\,\sum_{(i,j)\in\cE_\circ'} \Bar{z}_{ij}^n\,,
\end{equation}
and
\begin{equation}\label{PL5.5F}
- \sum_{i\in\cI_\circ}\psi'_\veo(\Hat{x}_i)\Hat{x}_i
\,\le\,
\sum_{\Hat{x}_i^-\le \frac{1}{2m}\sqrt n\,\min_{i\sim j}\,\Bar{z}^n_{ij}} \Hat{x}_i^-
\,\le\, \frac{1}{2}\sqrt n\,\min_{i\sim j}\,\Bar{z}^n_{ij}\,.
\end{equation}
Therefore, by \cref{PL5.5B}, \cref{PL5.5D,PL5.5E,PL5.5F,R5.2}, we obtain
\begin{equation*}
\begin{aligned}
F_2^n(\Hat{x},\Hat{z})
&\,\le\, \frac{\veo}{2}
 \Biggl(\Hat{\vartheta}^n - \sum_{\ell\in\cI_\circ} \Hat{x}_\ell^-\Biggr)\\
&\,\le\, \frac{\veo}{2}
(1-\varkappa_\circ^n)\,\widetilde\varkappa_\circ^n\,\norm{\Hat{x}^-}^{}_1
\qquad\forall\, (\Hat{x},\Hat{z})\in
\bigl(\widetilde{\sX}^n\cap\cK_0^+)\times \hcZn(\Hat{x})\,.
\end{aligned}
\end{equation*}
The rest follows as in \cref{L5.4}.
\end{proof}

%
%
%

\section*{Acknowledgment}
This work is  supported in part by the Army Research Office 
through grant W911NF-17-1-0019, and
in part by NSF grants DMS-1715210, CMMI-1635410, and DMS/CMMI-1715875,
and in part by the Office of Naval Research through grant N00014-16-1-2956
and was approved for public release under DCN \#43-5454-19.

\bibliographystyle{apalike}
\bibliography{N-paper-Ref}

\end{document}